\documentclass[11pt]{amsart}
\usepackage[utf8]{inputenc}
\usepackage[T1]{fontenc}
\usepackage{lmodern}
\usepackage{amsmath,amssymb,amsthm}
\usepackage{commath,mathtools}
\usepackage{xcolor}
\usepackage{mathrsfs,dsfont}
\usepackage[right,nolabel]{showlabels}
\usepackage{enumitem}
\usepackage{chngcntr}
\usepackage{marginnote}
\usepackage{doi,hyperref}
\usepackage{import}
\usepackage{floatrow}
\usepackage{caption}
\usepackage{lipsum}

\newcommand{\cB}{\mathcal{B}}
\newcommand{\cT}{\mathcal{T}}

\newcommand{\cA}{\mathcal{Y}}
\newcommand{\cR}{\mathcal{R}}

\newcommand{\cG}{\mathcal{G}}
\newcommand{\ELL}{\mathscr{L}}
\newcommand{\cK}{\mathsf{Gap}}
\newcommand{\DD}{\mathcal{D}}

\newcommand{\R}{\mathbb{R}}
\newcommand{\Z}{\mathbb{Z}}

\newcommand{\E}{\mathcal{E}}
\newcommand{\EE}{\widetilde{\mathcal{E}}_J}

\newcommand{\BCE}{\mathsf{BCE}}

\newcommand{\Bad}{\mathsf{Bad}}
\newcommand{\Ch}{\mathsf{Ch}}

\newcommand{\Tree}{\mathsf{Tree}}
\newcommand{\Top}{\mathsf{Top}}

\newcommand{\one}{\mathds{1}}

\newcommand\restr[2]{{
		\left.\kern-\nulldelimiterspace 
		#1 
		\right|_{#2} 
}}

\newcommand{\HH}{\mathcal{H}^1}
\newcommand{\TT}{\mathbb{T}}

\DeclareMathOperator{\lip}{Lip}
\DeclareMathOperator{\spn}{span}
\DeclareMathOperator{\diam}{diam}
\DeclareMathOperator{\dist}{dist}
\DeclareMathOperator{\Fav}{Fav}

\newtheorem{theorem}{Theorem}
\newtheorem{theoremalph}{Theorem}
\newtheorem*{theorem*}{Theorem}
\newtheorem{lemma}[theorem]{Lemma}
\newtheorem*{lemma*}{Key Geometric Lemma}
\newtheorem{cor}[theorem]{Corollary}
\newtheorem{prop}[theorem]{Proposition}
\newtheorem{conjecture}[theorem]{Conjecture}
\newtheorem{question}[theorem]{Question}
\theoremstyle{remark}
\newtheorem{definition}[theorem]{Definition}
\newtheorem*{definition*}{Definition}
\newtheorem{remark}[theorem]{Remark}
\newtheorem{observation}[theorem]{Observation}

\numberwithin{equation}{section}
\numberwithin{theorem}{section}
\counterwithin{figure}{section}
\title[Quantitative Besicovitch projection theorem...]{Quantitative Besicovitch projection theorem for irregular sets of directions}


\author[D. D\k{a}browski]{Damian D\k{a}browski}
\address{Department of Mathematics and Statistics\\ University of Jyv\"askyl\"a,
	P.O. Box 35 (MaD)\\
	FI-40014 University of Jyv\"askyl\"a\\
	Finland}
\email{damian.m.dabrowski@jyu.fi}

\begin{document}
	\maketitle
	\begin{abstract}
		The classical Besicovitch projection theorem states that if a planar set $E$ with finite length is purely unrectifiable, then almost all orthogonal projections of $E$ have zero length. We prove a quantitative version of this result: if $E\subset\R^2$ is AD-regular and there exists a set of direction $G\subset \mathbb{S}^1$ with $\mathcal{H}^1(G)\gtrsim 1$ such that for every $\theta\in G$ we have $\|\pi_\theta\mathcal{H}^1|_E\|_{L^{\infty}}\lesssim 1$, then a big piece of $E$ can be covered by a Lipschitz graph $\Gamma$ with $\lip(\Gamma)\lesssim 1$. The main novelty of our result is that the set of good directions $G$ is assumed to be merely measurable and large in measure, while previous results of this kind required $G$ to be an arc.
		
		As a corollary, we obtain a result on AD-regular sets which avoid a large set of directions, in the sense that the set of directions they span has a large complement. It generalizes the following easy observation: a set $E$ is contained in some Lipschitz graph if and only if the complement of the set of directions spanned by $E$ contains an arc.
	\end{abstract}
	\tableofcontents

	\section{Introduction}
	\subsection{Besicovitch projection theorem}
	A Borel set $E\subset\R^2$ is said to be \emph{purely unrectifiable} if for any (1-dimensional) Lipschitz graph $\Gamma\subset \R^2$ we have
	\begin{equation*}
		\HH(E\cap\Gamma)=0.
	\end{equation*}

	One of the fundamental results of geometric measure theory is the Besicovitch projection theorem, which states that if $E\subset\R^2$ is purely unrectifiable and $\HH(E)<\infty$, then almost all orthogonal projections of $E$ have zero length. We reformulate this result below in a way that is more suitable for the purpose of this article. 
	
	Let $\TT\coloneqq\R/\Z$, and for $\theta\in\TT$ we set $e_\theta\coloneqq (\cos(2\pi\theta),\,\sin(2\pi\theta)),$ and $\pi_\theta(x) \coloneqq e_\theta \cdot x$, so that $\pi_\theta:\R^2\to\R$ is the orthogonal projection map to the line $\ell_\theta\coloneqq\spn(e_\theta)$.
	
	\begin{definition}
		Given a Borel set $E\subset\R^2$, we define its \emph{Favard length} (also known as its \emph{Buffon's needle probability}) as
		\begin{equation*}
		\Fav(E) = \int_0^1 \HH(\pi_\theta(E))\, d\theta.
		\end{equation*}
	\end{definition}
	\begin{theoremalph}[\cite{besicovitch1939on}]\label{thm:BesFed}
		Let $E\subset\R^2$ be an $\HH$-measurable set with $0<\HH(E)<\infty$. Suppose that $\Fav(E)>0$.
		Then, there exists a Lipschitz graph $\Gamma$ such that
		\begin{equation*}
			\HH(\Gamma\cap E)>0.
		\end{equation*}
	\end{theoremalph}
	The planar result stated above is due to Besicovitch \cite{besicovitch1939on}, see \cite[Theorem 18.1]{mattila1999geometry} for a modern reference. A higher dimensional counterpart of \thmref{thm:BesFed}, dealing with $n$-dimensional subsets of $\R^d$, was shown by Federer \cite{federer1947varphi}, see also an alternative proof due to White \cite{white1998newproof}. In this paper we will only be concerned with 1-dimensional subsets of $\R^2$.
	
	Note that \thmref{thm:BesFed} is a purely qualitative result: it gives no estimate on the size of $\HH(\Gamma\cap E)$, nor on the Lipchitz constant of $\Gamma$. In the last thirty years many classical definitions and results of geometric measure theory have been quantified (see e.g. \cite{jones1990rectifiable,david1991singular,david1993analysis,azzam2015characterization,tolsa2014rectifiability,tolsa2017rectifiable}), finding applications in PDEs and harmonic analysis (see e.g. \cite{david1998unrectifiable,tolsa2003painleve,tolsa2005bilipschitz,nazarov2014on,azzam2016rectifiability,azzam2020harmonic}). However, obtaining a quantitative counterpart to \thmref{thm:BesFed} proved to be a notoriously difficult problem. Beyond its intrinsic appeal, this question is closely related to \emph{Vitushkin's conjecture}, which we briefly discuss in Subsection \ref{sec:Vitushkin}.
	
	The problem of quantifying \thmref{thm:BesFed} has seen a number of breakthroughs in the last few years \cite{martikainen2018characterising, chang2017analytic, orponen2020plenty}, which we will discuss shortly. In this article we make further progress on this question.
	
	\subsection{Quantifying Besicovitch projection theorem} In order to state our result, we need to quantify the finite length assumption of \thmref{thm:BesFed}.
	\begin{definition}
		We say that a set $E\subset\R^2$ is Ahlfors-David-regular, or AD-regular, if $E$ is closed and there exists a constant $C\ge 1$ such that for all $x\in E$ and $0<r<\diam(E)$
		\begin{equation*}
			C^{-1} r\le\HH(E\cap B(x,r))\le C r.
		\end{equation*}
		We will say that $E$ is AD-regular with constant $C_0$ if the inequality above holds with $C=C_0$.
	\end{definition}
	
	The following conjecture, if true, would be a very satisfactory quantitative version of the Besicovitch projection theorem.
	\begin{conjecture}\label{conj:1}
		Let $s \in (0,1),\ C_0\in (1,\infty)$, and let $E\subset \R^2$ be a bounded AD-regular set with constant $C_0$. Suppose that
		\begin{equation}\label{eq:big Fav}
		\Fav(E)\ge s\diam(E).
		\end{equation}
		Then, there exists a Lipschitz graph $\Gamma\subset\R^2$ with $\lip(\Gamma)\lesssim_{s,C_0} 1$ and
		\begin{equation*}
			\HH(\Gamma\cap E)\gtrsim_{s,C_0} \HH(E).
		\end{equation*}
	\end{conjecture}
	
	\begin{remark}
		A weaker version of Conjecture \ref{conj:1} was stated by David and Semmes in 1993 \cite{david1993quantitative}, and very recently proved by Orponen \cite{orponen2020plenty}. This is \thmref{thm:orpoPBP} discussed below.
	\end{remark}
	\begin{remark}
		The AD-regularity assumption in Conjecture \ref{conj:1} cannot be dropped nor replaced by the weaker assumption $\HH(E)\sim \diam(E)$, see \cite[Proposition 6.1]{chang2022structure}.
	\end{remark}
	\begin{remark}\label{rem:godddirs}
		Observe that the assumption \eqref{eq:big Fav} implies that there exists an $\HH$-measurable set $G\subset\TT$ with $\HH(G)\gtrsim s$ such that
		\begin{equation}\label{eq:big projections}
		\HH(\pi_\theta(E))\gtrsim s\diam(E)\quad\text{for all $\theta\in G$.}
		\end{equation}
		That is, $\Fav(E)\ge s\diam(E)$ implies that there exists a big set $G$ of ``good directions'' where $E$ has big projections. 
		
		On the other hand, the existence of a set $G$ as above implies that $\Fav(E)\gtrsim s^2\diam(E)$. Hence, the two conditions are equivalent, up to a constant. We stress that, a priori, the set of good directions $G$ arising from \eqref{eq:big Fav} is only measurable and large in measure. In particular, we have no lower bound on the size of the smallest interval contained in $G$. Even worse, it may be ``irregular'' in the sense that it is scattered inside $\TT$ and contains no interval.
	\end{remark}

	Significant progress towards proving Conjecture \ref{conj:1} has been recently achieved by Martikainen and Orponen \cite{martikainen2018characterising} and in the aforementioned work of Orponen \cite{orponen2020plenty}. We make further progress by proving the following result.
	\begin{theorem}\label{thm:main thm}
		Let $s \in (0,1), C_0,M\in (1,\infty)$, and let $E\subset\R^2$ be a bounded AD-regular set with constant $C_0$. Set $\mu=\HH|_{E}$.		
		Assume that there exists an $\HH$-measurable set $G\subset \TT$ with $\HH(G)\ge s$ and such that
			\begin{equation}\label{eq:projbdd}
			\|\pi_{\theta}\mu\|_{L^\infty(\R)}\le M\quad\text{for all $\theta\in G$,}
			\end{equation}
		where $\pi_\theta\mu$ is the push-forward of $\mu$ by $\pi_\theta$. 
		
		Then, there exists a Lipschitz graph $\Gamma\subset\R^2$ with $\lip(\Gamma)\lesssim_{C_0,M} 1$ and
		\begin{equation*}
		\HH(\Gamma\cap E)\gtrsim_{s,C_0,M} \HH(E).
		\end{equation*}
	\end{theorem}
	Note that the $L^\infty$-condition \eqref{eq:projbdd} implies the big projections condition \eqref{eq:big projections}:
	\begin{equation*}
	\HH(\pi_\theta(E))\ge M^{-1}\mu(E)\gtrsim M^{-1}C_0^{-1}\diam(E),
	\end{equation*}
	but in general \eqref{eq:projbdd} is much stronger than \eqref{eq:big projections}.
	
	\begin{remark}\label{rem:novelty}
		The main novelty of \thmref{thm:main thm} is that it allows us to work with a set of directions $G\subset\TT$ which is merely $\HH$-measurable and large in measure, just like the set of good directions arising from Conjecture \ref{conj:1} (see \remref{rem:godddirs}). Previous results of this type, which we discuss below, needed to assume something about projections in a large \emph{interval} of directions. Just how big of a difference this makes is discussed further in \remref{rem:novelty2}.
	\end{remark}
	
	\subsection{Comparison with results of Martikainen and Orponen}
	Let us compare \thmref{thm:main thm} with the results from \cite{martikainen2018characterising} and \cite{orponen2020plenty}. We only state their planar versions for simplicity, but both have higher-dimensional counterparts.
	\begin{theoremalph}[{\cite{martikainen2018characterising}}]\label{thm:martikainenorpo}
		Let $s \in (0,1), C_0,M\in (1,\infty)$, and let $E\subset\R^2$ be an AD-regular set with constant $C_0$. Let $E_1\subset E\cap B(0,1)$ be an $\HH$-measurable subset with $\HH(E_1)\ge s$. Set $\mu=\HH|_{E_1}$.
		
		Assume there exists $\theta_0\in \TT$ such that for $G=(\theta_0,\, \theta_0+s)$ we have
		\begin{equation}\label{eq:projbddL2}
			\int_{G}\|\pi_{\theta}\mu\|_{L^2(\R)}^2\, d\theta\le M.
		\end{equation}
		Then, there exists a Lipschitz graph $\Gamma\subset\R^2$ with $\lip(\Gamma)\lesssim_{s,C_0,M} 1$ and
		\begin{equation*}
			\HH(\Gamma\cap E_1)\gtrsim_{s,C_0,M} \HH(E_1).
		\end{equation*}
	\end{theoremalph}

	The result below was conjectured in \cite{david1993quantitative}, and it was proved very recently by Orponen.
	
	\begin{theoremalph}[\cite{orponen2020plenty}]\label{thm:orpoPBP}
		Let $s \in (0,1), C_0\in (1,\infty)$, and let $E\subset\R^2$ be an AD-regular set with constant $C_0$. Suppose that for every $x\in E$ and $0<r<\diam(E)$ there exists $\theta_{x,r}\in\TT$ such that for all $\theta\in G_{x,r}=(\theta_{x,r},\, \theta_{x,r}+s)$ we have
		\begin{equation}\label{eq:bigproj}
			\HH(\pi_\theta(E\cap B(x,r)))\ge s r.
		\end{equation}
		Then, for every $x\in E$ and $0<r<\diam(E)$ there exists a Lipschitz graph $\Gamma_{x,r}\subset\R^2$ with $\lip(\Gamma_{x,r})\lesssim_{s,C_0} 1$ and
		\begin{equation*}
			\HH(\Gamma_{x,r}\cap E\cap B(x,r))\gtrsim_{s,C_0} \HH(E\cap B(x,r)).
		\end{equation*}		
	\end{theoremalph}
	Observe that none of the three results above (\thmref{thm:main thm}, \thmref{thm:martikainenorpo}, \thmref{thm:orpoPBP}) implies any other, at least not in an obvious way. We summarize the main differences between them below. 
	
		Firstly, as already mentioned in \remref{rem:novelty}, in all three results we assume that $\HH(G)\ge s$, but in \thmref{thm:main thm} we only assume that $G$ is $\HH$-measurable, whereas in the other two results we assume that $G$ is an interval. We achieved this improvement at the cost of assuming better regularity of $\pi_{\theta}\mu$ for each $\theta\in G$ than in either \thmref{thm:martikainenorpo} or \thmref{thm:orpoPBP}, compare \eqref{eq:projbdd} with \eqref{eq:projbddL2} and \eqref{eq:bigproj}.
		
		Secondly, observe that \thmref{thm:main thm} and \thmref{thm:martikainenorpo} are ``single-scale results'', whereas \thmref{thm:orpoPBP} is a ``multi-scale result'', in the sense that in \thmref{thm:orpoPBP} one needs to assume that $E$ has big projections at \emph{all scales and locations} in order to get Lipschitz graphs covering $E$. Obtaining a single-scale version of \thmref{thm:orpoPBP} is an open problem stated in \cite[Question 1]{orponen2020plenty}.
		
		Finally, \thmref{thm:martikainenorpo} holds for {large subsets} of AD-regular sets, whereas \thmref{thm:main thm} and \thmref{thm:orpoPBP} have only been proven for AD-regular sets.
		
		
	
	\subsection{Related results} In \cite{david1993quantitative} David and Semmes proved that if $E\subset\R^2$ is AD-regular, it satisfies \emph{the weak geometric lemma} (a multi-scale flatness property), and $\HH(\pi_\theta(E))\gtrsim 1$ for some $\theta\in \TT$ (a single direction is enough!), then $E$ contains a big piece of a Lipschitz graph.
	
	In \cite{jones1997checkerboards} the authors proved a quantitative Besicovitch projection theorem for sets $E$ which are boundaries of open sets. The structure of sets with nearly maximal Favard length was studied in \cite{chang2022structure}. A version of Besicovitch projection theorem for Radon measures was recently shown in \cite{tasso2022rectifiability}. A version of the Besicovitch projection theorem for metric spaces was proved in \cite{bate2020purely}.
	
	See \cite{chang2017analytic,dabrowski2020cones} for the study of \emph{conical energies}, which we also use in the proof of \thmref{thm:main thm}. Closely related concepts of \emph{conical defect} and \emph{measures carried by Lipschitz graphs} were studied in \cite{badger2020radon}. 
	
	An alternative approach to quantifying Besicovitch projection theorem is to estimate the {rate of decay of Favard length} of $\delta$-neighbourhoods of certain purely unrectifiable sets. See \cite{mattila1990orthogonal,peres2002howlikely,tao2009quantitative,laba2010favard,bateman2010estimate,bond2010buffon,nazarov2011power,bond2014buffon,laba2014recent,wilson2017sets,bongers2019geometric,laba2022vanishing}.
	
	The Besicovitch projection theorem, and some of the results mentioned above, have been also proven for \emph{generalized projections} in place of orthogonal projection. See \cite{hovila2012besicovitch,bond2011circular,cladek2020upper,bongers2021transversal,davey2022quantification}.
	
	\subsection{Vitushkin's conjecture}\label{sec:Vitushkin}
	One of the main motivations for the study of Conjecture \ref{conj:1} is to complete the solution to Vitushkin's conjecture, which asks for the relation between Favard length and analytic capacity. Different parts of the conjecture have been verified or disproved in \cite{calderon1977cauchy,david1998unrectifiable,mattila1986smooth,jones1988positive}, but one question remains: given a $1$-dimensional compact set $E\subset\R^2$ with non-$\sigma$-finite length and $\Fav(E)>0$, is the analytic capacity of $E$ positive? It is beyond the scope of this introduction to discuss this in detail, but let us mention that recent progress on this problem made in \cite{chang2017analytic} and \cite{dabrowski2022analytic} used the ideas and results obtained in \cite{martikainen2018characterising} and \cite{orponen2020plenty}, respectively. Solving Conjecture \ref{conj:1} (or even it's weaker, multi-scale version) would immediately mark substantial progress on this question, see \cite[Remark 1.9]{dabrowski2022analytic}. We refer the interested reader to \cite{dabrowski2022analytic} for details.
	
	\subsection{Directions spanned by sets} We give an application of \thmref{thm:main thm} to directions spanned by sets.
	\begin{definition}
		Given a Borel set $E\subset\R^2$ we define \emph{the set of directions spanned by $E$} as
		\begin{equation*}
		D(E) \coloneqq \bigg\{\frac{x-y}{|x-y|}\ :\ x,y\in E,\ x\neq y\bigg\}\subset \mathbb{S}^1,
		\end{equation*}
		or, using our preferred parametrization of the circle,
		\begin{equation*}
		D_{\TT}(E) \coloneqq \frac{1}{2\pi}\arg(D(E))\subset\TT.
		\end{equation*}
		We will denote the complement of $D_{\TT}(E)$ by $G_\TT(E)$, and we will say that the directions in $G_\TT(E)$ are \emph{avoided} by $E$.
	\end{definition}
	
	Sets of directions spanned by subsets of $\R^d$ have been studied in \cite{orponen2011radial,iosevich2012onsets}. They are closely related to \emph{radial projections} due to the fact that
	\begin{equation*}
		D(E) = \bigcup_{x\in E} \pi_x(E\setminus\{x\}),
	\end{equation*}
	where $\pi_x(y) = \frac{x-y}{|x-y|}$ is the radial projection map from $x$. The behaviour of purely unrectifiable sets under radial projections was studied in \cite{marstrand1954some,simon2006visibility,bond2016quantitative}. See also \cite{mattila1981integralgeometric,csornyei2000onthevisibility,csornyei2001howto,vardakis2022geometry,dote2022exceptional,orponen2022kaufman}.
	\begin{remark}\label{rem:empty cone}
		Given $G\subset\TT$ and $x\in\R^2$, consider the cone $X(x,G) \coloneqq \bigcup_{\theta\in G}\ell_{x,\theta},$ where $\ell_{x,\theta}=x+\spn(e_\theta)$. Note that if $E\subset\R^2$ satisfies $G_\TT(E)\neq\varnothing$, then
		\begin{equation*}
			E\cap X(x,G_{\TT}(E))=\{x\}\quad\text{for all $x\in E$},
		\end{equation*}
		and $G_\TT(E)$ is the largest subset of $\TT$ with this property.
	\end{remark}
	
	The following is an easy observation used in many geometric measure theory proofs (for example, in the proof of \thmref{thm:BesFed}).
	\begin{observation}\label{obs:trivial}
		A set $E\subset\R^2$ is contained in some Lipschitz graph $\Gamma\subset\R^2$ if and only if there exists a (non-degenerate) interval $I\subset\TT$ such that
		\begin{equation*}
		I\subset G_\TT(E).
		\end{equation*}
		Furthermore, we have $\lip(\Gamma)\lesssim \HH(I)^{-1}$. Usually this result is stated in terms of the ``empty cone condition''
		\begin{equation*}
			E\cap X(x,I)=\{x\}\quad\text{for all $x\in E$},
		\end{equation*}
		but this is equivalent by \remref{rem:empty cone}. See \cite[Lemma 15.13]{mattila1999geometry} or \cite[Remark 1.11]{martikainen2018characterising} for an easy proof.
	\end{observation} 	
	It is natural to ask if the following generalization of the observation above is true:
	\begin{question}\label{q:nontriv}
		Let $s\in (0,1),\ C_0\ge 1$. Suppose that $E\subset\R^2$ is a bounded AD-regular set with constant $C_0$, and that
		\begin{equation*}
		\HH(G_\TT(E))\ge s.
		\end{equation*}
		Is it possible to find a Lipschitz graph $\Gamma\subset\R^2$ with $\lip(\Gamma)\lesssim_{s,C_0} 1$ and 
		\begin{equation*}
		\HH(\Gamma\cap E)\gtrsim_{s,C_0} \HH(E)?
		\end{equation*}
	\end{question}
	\begin{remark}\label{rem:novelty2}
	Note that in Question \ref{q:nontriv} we added many assumptions compared to Observation \ref{obs:trivial}, we weakened the conclusion, and the only assumption that is weaker in Question \ref{q:nontriv} is that we assume no additional structure on $G_\TT(E)$ beyond large $\HH$-measure. This makes all the difference: the case of a big interval, as in Observation \ref{obs:trivial}, is very easy, whereas Question \ref{q:nontriv} appears to be non-trivial.	
	Similarly, the fact that \thmref{thm:main thm} does not assume much regularity about the set of good directions $G$ leads to genuinely new difficulties compared to \thmref{thm:martikainenorpo} and \thmref{thm:orpoPBP}, and it is not merely a cosmetic difference. 
	\end{remark}
	
	Using \thmref{thm:main thm} we are able to answer affirmatively the following special case of Question \ref{q:nontriv}.
	\begin{cor}
		Let $s\in (0,1),\ C_0\ge 1$. Suppose that $E\subset \R^2$ is a bounded AD-regular set with constant $C_0$, and that
		\begin{equation*}
		\HH(G_\TT(E))\ge s.
		\end{equation*}
		Suppose further that $E$ is a union of parallel line segments.
		Then, there exists a Lipschitz graph $\Gamma\subset\R^2$ with $\lip(\Gamma)\lesssim_{s,C_0} 1$ and 
		\begin{equation*}
		\HH(\Gamma\cap E)\gtrsim_{s,C_0} \HH(E).
		\end{equation*}
	\end{cor}
	\begin{proof}
		Let $\theta_0\in \TT$ be such that the line segments comprising $E$ are parallel to $\ell_{\theta_0}$. Set
		\begin{equation*}
		G\coloneqq G_\TT(E)\setminus (\theta_0 - 0.1s,\theta_0+0.1s).
		\end{equation*}
		Let $\theta\in G$ and $y\in\pi_\theta(E)$. Since $E$ avoids the direction $\theta$, we get that $E$ is a graph over $\ell_\theta^\perp$, and it consists of segments forming angle $\measuredangle(\ell_{\theta_0}, \ell_\theta)\sim |\theta-\theta_0|$ with $\ell_{\theta}=(\ell_\theta^\perp)^\perp$. It follows that
		\begin{equation*}
		\pi_\theta^\perp \HH|_E(y) = \lim_{h\to 0} \frac{\HH(E\cap (\pi_{\theta}^\perp)^{-1}((y-h, y+h))}{h} \lesssim \lim_{h\to 0} \frac{|\theta-\theta_0|^{-1}h}{h} \lesssim s^{-1}.
		\end{equation*}
		Hence, $\|\pi_{\theta}^\perp \HH|_E\|_{\infty}\lesssim s^{-1}$. Since
		\begin{equation*}
		\HH(G)\ge \HH(G_\TT(E)) - 0.2 s\ge \frac{s}{2},
		\end{equation*}
		we may apply \thmref{thm:main thm} (with $G^\perp$ instead of $G$) to find the desired Lipschitz graph $\Gamma$ with $\lip(\Gamma)\lesssim_{s,C_0} 1$ and $\HH(\Gamma\cap E)\gtrsim_{s,C_0} \HH(E)$.
	\end{proof}

	We mention another interesting question in the same vein, which is essentially a qualitative version of Question \ref{q:nontriv}. 
	
	It follows from the definition of purely unrectifiable sets and Observation \ref{obs:trivial} that if $E$ is purely unrectifiable and $\HH(E)>0$, then $D_\TT(E)$ is dense in $\TT$. What can be said about $\HH(D_\TT(E))$?
	\begin{question}
		Suppose that $E\subset\R^2$ is purely unrectifiable, and $0<\HH(E)<\infty$. Do we have
		\begin{equation*}
		\HH(D_{\TT}(E))=\HH(\TT)?
		\end{equation*}
	\end{question}
	The answer is yes for \emph{homogeneous sets} (examples of which include self-similar sets satisfying the strong separation condition for which the linear parts of the similarities contain no rotations) by \cite[Proposition 3.1]{rossi2019holder}; in fact, for such sets Rossi and Shmerkin proved that $D_\TT(E)=\TT$. To the best of our knowledge, the question is open for general purely unrectifiable sets. Up until recently it wasn't even clear if $\dim_H(D_{\TT}(E))=1$, but this follows from a recent paper of Orponen, Shmerkin, and Wang \cite{orponen2022kaufman}.
	
	\subsection{Plan of the article} In Section \ref{sec:sketch} we sketch the proof of \thmref{thm:main thm}. In Section \ref{sec:prel} we introduce some notation, list all the parameters appearing in the proof, and remind some useful results from \cite{chang2017analytic} and \cite{dabrowski2020cones}. In Section \ref{sec:mainprop} we state our main proposition, \propref{prop:main prop}, and we show how it can be used to prove \thmref{thm:main thm}. We prove the main proposition in Sections \ref{sec:rectangles}--\ref{sec:keygeometric}.
	In Section \ref{sec:rectangles} we introduce a ``dyadic grid of rectangles'' adapted to \propref{prop:main prop}, and we prove some basic measure estimates on these rectangles. Section \ref{sec:conical} contains a stopping time argument and a corona decomposition involving conical energies. In Sections \ref{sec:interior energy}--\ref{sec:keygeometric} we estimate these energies. Finally, in Appendix \ref{sec:appendix} we prove one of the results from Section \ref{sec:prel}.
	
	\subsection*{Acknowledgments} 
	I am grateful to Alan Chang, Tuomas Orponen, Xavier Tolsa, and Michele Villa for inspiring discussions.
	
	I was supported by the Academy of Finland via the projects \emph{Incidences on Fractals}, grant No. 321896, and \emph{Quantitative rectifiability and harmonic measure beyond the Ahlfors-David-regular setting}, grant No. 347123.
	
	\section{Sketch of the proof}\label{sec:sketch}
	Suppose that $E\subset\R^2$ is bounded and AD-regular, $\mu=\HH|_E$, $G\subset\TT$ satisfies $\HH(G)\gtrsim 1$, and for all $\theta\in G$ we have $\|\pi_\theta\mu\|_\infty\lesssim 1$. Using \propref{prop:changtolsa}, which is a result from \cite{chang2017analytic}, it is easy to show that this implies
	\begin{equation}\label{eq:13}
	\int_{\R^2}\int_0^{\diam(E)}\frac{\mu(X(x,G^\perp,r))}{r}\, \frac{dr}{r}d\mu(x)\lesssim \mu(E),
	\end{equation}
	where $X(x,G^\perp,r)= X(x,G^\perp)\cap B(x,r)$, and $X(x,G^\perp)$ is the union of lines passing through $x$ with directions perpendicular to those from $G$. See \secref{sec:notation} for the precise definition.
	
	Estimate \eqref{eq:13} is reminiscent of \propref{prop:conical energy}, which was observed in \cite{dabrowski2020cones} but is essentially due to \cite{martikainen2018characterising}. This result says that if the estimate \eqref{eq:13} holds with $G$ which is a large interval, then one can find a big piece of a Lipschitz graph inside $E$. The problem is, the set $G$ given by \thmref{thm:main thm} may be a very complicated set, possibly consisting of many tiny intervals, or not containing any intervals at all.
	
	This issue is addressed by our main proposition, \propref{prop:main prop}. Roughly speaking, it says that if we start with a set of ``good directions'' $G_J$ which almost fills an interval $J$, then the goodness of $G_J$ propagates to all of $J$, and even to the enlarged interval $3J$. More precisely, given an interval $J\subset\TT$, possibly very short, and a set $G_J\subset J$ with $\HH(J\setminus G_J)\le \varepsilon\HH(J)$, where $\varepsilon>0$ is very small, and under some additional technical assumptions involving $\|\pi_\theta\mu\|_{\infty}$, one has
	\begin{multline}\label{eq:14}
	\int_{E}\int_0^{\diam(E)} \frac{\mu(X(x,3J,r))}{r}\, \frac{dr}{r}d\mu(x)\\
	\le C_{\mathsf{Prop}}\bigg(\int_{E}\int_0^{\diam(E)} \frac{\mu(X(x,G_J,r))}{r}\, \frac{dr}{r}d\mu(x) + \HH(J)\mu(E)\bigg).
	\end{multline}
	 Crucially, the constants $\varepsilon$ and $C_{\mathsf{Prop}}$ do not depend on $\HH(J)$.
	
	Using the idea of the good set $G$ propagating and becoming larger, we are able to apply \propref{prop:main prop} iteratively, so that after a bounded number of iterations we end up with an estimate \eqref{eq:13} with the set $G$ replaced by some interval $J_0$ with $\HH(J_0)\sim 1$. This allows us to use \propref{prop:conical energy} to obtain a big piece of Lipschitz graph inside $E$. All of this is done in Section \ref{sec:mainprop}, assuming that \propref{prop:main prop} is true. The remainder of the paper is dedicated to the proof of \propref{prop:main prop}.
	
	In Section \ref{sec:rectangles} we consider a ``dyadic lattice of rectangles'' $\DD=\bigcup_k\DD_k$, where each $\DD_k$ is a partition of $E$. The rectangles we work with have a very large, but fixed, aspect ratio equal to $\HH(J)^{-1}$, and they all point in the same direction, corresponding to the mid-point of $J$. A priori, the fact that $\mu$ is AD-regular only tells us that a rectangle $Q\in\DD$ satisfies
	\begin{equation*}
	\ell(Q)\lesssim \mu(Q)\lesssim \HH(J)^{-1}\ell(Q),
	\end{equation*}
	where $\ell(Q)$ denotes the length of the shorter side of $Q$. This is no good: it is crucial that our estimates do not explode as $\HH(J)\to 0$. Luckily, due to one of the assumptions on $\|\pi_\theta\mu\|_{\infty}$, we show in \lemref{lem:ADRrectangles} that $\mu(Q)\sim \ell(Q)$. So in a sense, we need the $L^\infty$-norm in \eqref{eq:projbdd}, and not just the $L^2$-norm as in \thmref{thm:martikainenorpo}, to ensure that our rectangles are ``AD-regular''.
	
	In Section \ref{sec:conical} we introduce conical energies $\E_G(Q)$ and $\E_J(Q)$, associated to $G_J$ and $3J$, respectively. They are essentially local versions of double intergals from \eqref{eq:14}, so that
	\begin{equation*}
	\int_{\R^2}\int_0^{\diam(E)}\frac{\mu(X(x,G_J,r))}{r}\, \frac{dr}{r}d\mu(x)\sim \sum_{Q\in\DD}\E_G(Q)\mu(Q),
	\end{equation*}
	and an analogous estimate holds for $3J$ and $\E_J(Q)$.
	Inspired by \cite{chang2017analytic}, we conduct a stopping time argument and a corona decomposition of $\DD$ into a family of trees $\Tree(R),\ R\in\Top$. What we gain is that for any $R\in\Top$ and most $x\in R$ the cone $X(x,G_J)$ does not intersect $E$ at the scales associated to $\Tree(R)$.
	
	In Sections \ref{sec:interior energy} and \ref{sec:exterior energy} we prove that for any $R\in\Top$
	\begin{equation*}
	\sum_{Q\in\Tree(R)}\E_J(Q)\mu(Q)\lesssim \sum_{Q\in\Tree(R)}\E_G(Q)\mu(Q) + \HH(J)\mu(R),
	\end{equation*}
	which is enough to obtain \eqref{eq:14}. To prove the estimate above, we divide $\E_J(Q)$ into an ``interior'' conical energy $\E^{int}_J(Q)$ associated to $0.5J$, and an ``exterior'' conical energy $\E^{ext}_J(Q)$ associated to $3J\setminus 0.5J$. In Section \ref{sec:interior energy} we deal with the interior part. This is another important point where we use the technical assumptions related to $\|\pi_\theta\mu\|_\infty$: together with AD-regularity of $E$ they allow us to get a strong, pointwise estimate $\E^{int}_J(Q)\lesssim \E_G(Q)$. As a corollary, we get that for $R\in\Top$ and all $x\in R$ the cone $X(x,0.5J)$ does not intersect $E$ at the scales associated to $\Tree(R)$.
	
	Finally, in Section \ref{sec:exterior energy} we estimate the exterior energy $\E^{ext}_J(Q)$. The argument uses the key geometric lemma of this article, \lemref{lem:key geometric lemma}, which we prove in Section \ref{sec:keygeometric}. The proof is purely geometric, and we believe it is the true heart of this article. 
	
	A simplified version of \lemref{lem:key geometric lemma} says the following:
	\begin{lemma*}[simplified]
		Let $A\subset B(0,1)\subset\R^2$ be an AD-regular sets consisting of horizontal segments. Let $J\subset\TT$  be an interval such that $\HH(J)\le c$ for a small absolute constant $c>0$, and such that $X(0,J)$ contains the vertical axis. Assume that
		\begin{equation*}
		A\cap X(x,J)=\{x\}\quad\text{for every $x\in A$.}
		\end{equation*}
		Suppose that there is a point $y\in A$ and a scale $r\in (0,1)$ such that
		\begin{equation*}
		A\cap X(y, 3J,2r)\setminus B(y,r) \neq\varnothing.
		\end{equation*}
		Then, there exists an interval $K\subset\R$, which is a connected component of $\R\setminus\pi_0(A)$ (where $\pi_0$ is the projection to the horizontal axis), such that $\HH(K)\sim \HH(J)r$ and $\pi_0(y)\in CK$ for some absolute $C\ge 1$. 
	\end{lemma*} 	
	It is not too difficult to show using this lemma that a set $A$ as above satisfies
	\begin{equation*}
	\int_A\int_0^{\diam(A)}\frac{\HH(A\cap X(x,3J,r))}{r}\, \frac{dr}{r}d\HH(x) \lesssim \HH(J)\HH(A).
	\end{equation*}
	This is essentially where the last term in \eqref{eq:14} comes from.
	\section{Preliminaries}\label{sec:prel}
	\subsection{Notation}\label{sec:notation}
	Given $x\in\R^2$ and $\theta\in\TT$ we set
	\begin{align*}
	e_\theta &\coloneqq (\cos(2\pi\theta),\sin(2\pi\theta))\in\mathbb{S}^1,\\
	\pi_\theta (x)&\coloneqq e_\theta \cdot x,\\
	\ell_{x,\theta}&\coloneqq x+\spn(e_\theta),\\
	\ell_\theta&\coloneqq \ell_{0,\theta}.
	\end{align*}
	
	For $x\in\R^2$ and a measurable set $I\subset \TT$ we define the cone centered at $x$ with directions in $I$ as
	\begin{equation*}
	X(x,I) = \bigcup_{\theta\in I}\ell_{x,\theta}.
	\end{equation*}
	Note that we do not require $I$ to be an interval. We also set $I^\perp = I+1/4$.
	
	For $0<r<R$ we define truncated cones as
	\begin{gather*}
	X(x,I,r) = X(x,I)\cap B(x,r),\\
	X(x,I,r,R) = X(x,I,R)\setminus B(x,r).
	\end{gather*}
	
	In case $I = [\theta - a, \theta + a]$, we have an algebraic characterization of $X(x,I)$: $y\in X(x,I)$ if and only if
	\begin{equation}\label{eq:cone algebraic}
	|\pi^\perp_{\theta}(y)-\pi^\perp_{\theta}(x)|\le \sin(2\pi a)|x-y|.
	\end{equation}

	We will denote by $\Delta$ the usual family of half-open dyadic intervals on $[0,1)\simeq \TT$. If $J\in \Delta$, then $\Delta(J)$ denotes the collection of dyadic intervals contained in $J$. For $I\in\Delta\setminus\{[0,1)\}$, the notation $I^1$ will be used for the dyadic parent of $I$.
	
	Given an interval $I\subset\TT$ and $C>0$, we will write $CI$ to denote the interval with the same midpoint as $I$ and length $C\HH(I)$.
	
	The closure of a set $A$ will be denoted by $\overline{A}$, and its interior by $\mathrm{int}(A)$.
	\subsection{Constants and parameters} Whenever we write $f\lesssim g$, this should be understood as ``there exists an absolute constant $C>0$ such that $f\le Cg$.'' We will write $f\lesssim_A g$ if we allow the constant $C$ to depend on some parameter $A$. We also write $f\sim g$ to denote $g\lesssim f\lesssim g$, and similarly $f\sim_A g$ stands for $g\lesssim_A f\lesssim_A g$.
	
	Throughout the proof we use many constants and parameters. We list the most important ones here for reader's convenience. The notation $C_1=C_1(C_2)$ means ``$C_1$ is a parameter whose value depends on the value of parameter $C_2$''.
	\begin{itemize}
		\item $C_0\ge 1$ is the AD-regularity constant of the set $E$.
		\item $M\ge 1$ is the constant bounding the $L^\infty$-norm of projections in the assumptions of \thmref{thm:main thm} and \propref{prop:main prop}.
		\item $s\in (0,1)$ is the constant from the assumption $\HH(G)\ge s$ in \thmref{thm:main thm}.
		\item $\varepsilon=\varepsilon(C_0,M)\in(0,1)$ is a constant appearing in \propref{prop:main prop}, see \eqref{eq:Gmeas}. It is chosen in \lemref{lem:filling gaps}. One could take $\varepsilon=c C_0^{-1}M^{-1}$ for some small absolute $c\in (0,1)$.
		\item $C_{\mathsf{Prop}}=C_{\mathsf{Prop}}(C_0,M)>1$ is a big constant appearing in the conclusion of \propref{prop:main prop}.
		\item $c_1\in (0,1)$ is a small absolute constant appearing in the assumption $\HH(J)\le c_1 C_0^{-1} M^{-1}$ of \propref{prop:main prop}. It is fixed above \eqref{eq:Pitall}.
		\item $\rho=1/1000$ is the constant from \thmref{thm:dyadic cubes}, so that for $Q\in\DD_k$ we have $\ell(Q)=4\rho^k$.
		\item $A=A(C_0,M)\ge 1000$ is a large constant appearing in the definition of $\E_G(Q)$ \eqref{eq:energydef}. It is fixed in \lemref{lem:AinGX}, one could take $A=C C_0M$ for some absolute $C\ge 1000$. 
		\item $\delta=\delta(A,M,C_0)\in (0,1)$ is the $\BCE$-parameter, appearing in \eqref{eq:BCE def}. It is fixed in \lemref{lem:empty cones}.
		\item $N\sim C_0 M$ is a parameter appearing in the definition of rectangles $\cG_i$, below \eqref{eq:Ndef}. It's exact value is chosen in \lemref{lem:good-rectangle}.
	\end{itemize}
	
	\subsection{Useful results on cones and projections}
	We recall some results that will be useful in our proof. The proposition below is a simplified version of Corollary 3.3 from {\cite{chang2017analytic}}.
	\begin{prop}\label{prop:changtolsa}
		Let $\mu$ be a finite, compactly supported Borel measure on $\R^2$, and $I\subset \TT$ an open set. Then,
		\begin{equation*}
		\int_{\R^2} \int_0^\infty \frac{\mu(X(x,I,r))}{r}\, \frac{dr}{r}d\mu(x) \lesssim \int_{I}\|\pi_\theta^\perp \mu\|_2^2\, d\theta.
		\end{equation*}
	\end{prop}
	We remark that the estimate above is equality if $\mu$ is given by a Schwartz function, see Proposition 3.2 in \cite{chang2017analytic}. For general measures, a partial converse inequality can be found in Appendix A of \cite{chang2017analytic}. In this article we will only need the following corollary of Proposition \ref{prop:changtolsa}.
	\begin{cor}\label{cor:changtolsa}
		Let $E\subset \R^2$ and $G\subset\TT$ be as in \thmref{thm:main thm}, and let $\mu = \HH|_E$. Then,
		\begin{equation*}
			\int_{\R^2} \int_0^\infty \frac{\mu(X(x,G^\perp,r))}{r}\, \frac{dr}{r}d\mu(x) \lesssim M\HH(G)\mu(E),
		\end{equation*}
		where $G^\perp = G + 1/4$.
	\end{cor}
	If $G$ is open, then this follows almost immediately from \propref{prop:changtolsa}. The case of a general measurable set $G$ is a long and uninspiring exercise in measure theory, so we postpone it to the appendix.

	The following result is a simplified version of Proposition 10.1 from \cite{dabrowski2020cones}, which in turn is a consequence of Proposition 1.12 from \cite{martikainen2018characterising}.
	\begin{prop}\label{prop:conical energy}
		Let $E\subset\R^2$ be a bounded AD-regular set with constant $C_0$. Let $F\subset E$ be such that $\HH(F)\ge \kappa \HH(E)$.
		Assume there exists an interval $J\subset \TT$ with $\HH(J)=s$ such that for $\HH$-a.e. $x\in F$
		\begin{equation*}
			\int_0^1 \frac{\HH(X(x,J,r)\cap F)}{r}\, \frac{dr}{r}\le M.
		\end{equation*}
		Then, there exists a Lipschitz graph $\Gamma\subset\R^2$ with $\lip(\Gamma)\lesssim_{s} 1$ and
		\begin{equation*}
			\HH(F\cap \Gamma)\gtrsim_{C_0,s,M,\kappa} \HH(F).
		\end{equation*}
	\end{prop}
		
	\section{Main proposition and proof of \thmref{thm:main thm}}\label{sec:mainprop}
	The following is our main proposition.
	\begin{prop}\label{prop:main prop}
		Let $1\le C_0,M<\infty$. There exist constants $0<\varepsilon<1<C_{\mathsf{Prop}}<\infty$, which depend on $M, C_0$, such that the following holds.		
		Assume that:
		\begin{enumerate}[label=({\alph*})]
			\item $E\subset \R^2$ is a bounded AD-regular set with constant $C_0$, and set $\mu = \HH|_{E}$,
			\item $J\subset \TT$ is an interval with $\HH(J)\le c_1 C_0^{-1} M^{-1}$, where $c_1>0$ is a small absolute constant,
			\item there exists $\theta_0\in 3J$ such that $\|\pi_{\theta_0}^\perp\mu\|_\infty \le M$,
			\item $G\subset J$ is a closed set which satisfies
			\begin{equation}\label{eq:Gmeas}
				\HH(G)\ge (1-\varepsilon)\HH(J),
			\end{equation}
			\item for every interval $I$ which is a connected component of $J\setminus G$ there exists $\theta_I\in 3I$ such that $\|\pi_{\theta_I}^\perp\mu\|_\infty \le M$,
		\end{enumerate}		
		Then,
		\begin{multline*}
		\int_{E}\int_0^{\diam(E)} \frac{\mu(X(x,3J,r))}{r}\, \frac{dr}{r}d\mu(x)\\
		\le C_{\mathsf{Prop}}\bigg(\int_{E}\int_0^{\diam(E)} \frac{\mu(X(x,G,r))}{r}\, \frac{dr}{r}d\mu(x) + \HH(J)\mu(E)\bigg).
		\end{multline*}
	\end{prop}
	\begin{remark}
		In the proposition above, the interval $J$ may be open, closed, or half-open, it doesn't make a difference. In the conclusion we may take $3J$ to be a closed interval (in fact, the same proof gives the conclusion also with $CJ$ replacing $3J$, if we let $C_{\mathsf{Prop}}$ depend on $C$ as well, and as long as $\HH(CJ)\le c_1 C_0^{-1} M^{-1}$).
	\end{remark}
	We prove \propref{prop:main prop} in Sections \ref{sec:rectangles}--\ref{sec:keygeometric}.
	Now let us show how it can be used to prove \thmref{thm:main thm}. We begin by proving a corollary of \propref{prop:main prop}, which looks quite similar to \propref{prop:main prop} itself; the crucial difference is that it deals with sets $G\subset J$ with $\HH(G)<(1-\varepsilon)\HH(J)$. Recall that for a dyadic interval $I\in \Delta$ we denote by $I^1$ the dyadic parent of $I$.
	\begin{cor}\label{cor:main prop}
		Let $1\le C_0,M<\infty$. Let $\varepsilon=\varepsilon(M,C_0)$, $C_{\mathsf{Prop}}=C_{\mathsf{Prop}}(M,C_0)$ be as in \propref{prop:main prop}. 		
		Assume that:
		\begin{enumerate}[label=({\alph*})]
			\item $E\subset \R^2$ is a bounded AD-regular set with constant $C_0$, and $\mu = \HH|_{E}$,
			\item $J\subset \TT$ is a dyadic interval with $\HH(J)\le c_1C_0^{-1} M^{-1}$, where $c_1>0$ is as in \propref{prop:main prop},
			\item $G\subset \overline{J}$ is a finite union of closed dyadic intervals, which satisfies
			\begin{equation}\label{eq:Gassumption}
				0<\HH(G)< (1-\varepsilon)\HH(J),
			\end{equation}
			\item denoting the collection of maximal dyadic intervals contained in $J\setminus G$ by $\cB_{\Delta}$, for every $I\in\cB_\Delta$ there exists $\theta_I\in I^{1}$ such that $\|\pi_{\theta_I}^\perp\mu\|_\infty \le M$.
		\end{enumerate}
		Then, there exists a closed set $G_*$ with
		\begin{equation}\label{eq:GGstar}
		G\subset G_*\subset \overline{J},
		\end{equation}
		which is a finite union of closed dyadic intervals, such that
		\begin{equation}\label{eq:measu}
		\HH(G_*)\ge (1+{\varepsilon})\HH(G),
		\end{equation}
		and
		\begin{multline}\label{eq:ener}
		\int_E\int_0^{\diam(E)} \frac{\mu(X(x,G_*,r))}{r}\, \frac{dr}{r}d\mu(x)\\
		\le C_{\mathsf{Prop}}\bigg(\int_{E}\int_0^{\diam(E)} \frac{\mu(X(x,G,r))}{r}\, \frac{dr}{r}d\mu(x) + \HH(J)\mu(E)\bigg).
		\end{multline}
		Moreover, denoting by $\cB_{\Delta,*}$ the collection of maximal dyadic intervals contained in $J\setminus G_*$, we have
		\begin{equation}\label{eq:dyda}
		\cB_{\Delta,*}\subset \cB_{\Delta}.
		\end{equation}
	\end{cor}
	The statement above is quite involved, but it is very well-suited for its iterative application later on: note that the resulting set $G_*$ satisfies all the same assumptions as the set $G$ we started with, except perhaps for the measure assumption \eqref{eq:Gassumption}. 
	
	We divide the proof  of Corollary \ref{cor:main prop} into several steps.
	\subsubsection*{Definition of $G_*$}
		Let $\mathcal{I}\subset\Delta(J)$ be the family of maximal dyadic intervals such that for every $I\in \mathcal{I}$
		\begin{equation}\label{eq:Idef}
		\HH(I\cap G)\ge (1-\varepsilon)\HH(I).
		\end{equation}
		Since $G$ is a finite union of closed dyadic intervals, we get immediately that
		\begin{equation*}
		G\subset \bigcup_{I\in\mathcal{I}}\overline{I},
		\end{equation*}
		and that $\mathcal{I}$ is a finite family. Observe that the intervals in $\mathcal{I}$ are pairwise disjoint by maximality. Moreover, we have $J\notin\mathcal{I}$ due to \eqref{eq:Gassumption}, so that all $I\in\mathcal{I}$ are strictly contained in ${J}$.
		
		Consider the family $\mathcal{I}^1=\{I^1\}_{I\in\mathcal{I}}\subset\Delta(J)$, where $I^1$ denotes the dyadic parent of $I$, and let $\mathcal{I}_*$ be the family of maximal dyadic intervals from $\mathcal{I}^1$. The intervals in $\mathcal{I}_*$ are pairwise disjoint by maximality, and the family $\mathcal{I}_*$ is finite because $\mathcal{I}$ is finite.
		We set
		\begin{equation*}
		G_{*} \coloneqq \bigcup_{I\in\mathcal{I}_*}\overline{I}.
		\end{equation*}
		It remains to show that $G_*$ satisfies \eqref{eq:GGstar}, \eqref{eq:measu}, \eqref{eq:ener}, and \eqref{eq:dyda}.
	\begin{proof}[Proof of \eqref{eq:GGstar}]
		Note that
		\begin{equation*}
		G\subset \bigcup_{I\in\mathcal{I}}\overline{I}\subset \bigcup_{I\in\mathcal{I}}\overline{I^1}=\bigcup_{I\in\mathcal{I}_*}\overline{I} = G_*.
		\end{equation*}
		Since $\mathcal{I}_*\subset\Delta(J)$, we also have $G_*\subset \overline{J}$.
	\end{proof}	
	\begin{proof}[Proof of \eqref{eq:measu}]
		Recall that $\mathcal{I}$ was defined as the collection of maximal dyadic intervals where \eqref{eq:Idef} holds. Let $I\in\mathcal{I}_*$. We know that $I$ is a parent of some $I'\in\mathcal{I}$, and $I'$ is a maximal interval where \eqref{eq:Idef} holds. It follows that $I$ does not satisfy \eqref{eq:Idef}, which means that
		\begin{equation*}
		\HH(I\cap G)<(1-\varepsilon)\HH(I),
		\end{equation*}
		or equivalently,
		\begin{equation*}
		\HH(I\setminus G)\ge \varepsilon\HH(I).
		\end{equation*}
		Using this estimate we compute
		\begin{multline*}
		\HH(G_*) = \sum_{I\in\mathcal{I}_*}\HH(I) = \sum_{I\in\mathcal{I}_*}\HH(I\cap G) + \sum_{I\in\mathcal{I}_*}\HH(I\setminus G)\\
		= \HH(G) + \sum_{I\in\mathcal{I}_*}\HH(I\setminus G) \ge \HH(G) + \varepsilon\sum_{I\in\mathcal{I}_*}\HH(I)\\
		= \HH(G)+\varepsilon\HH(G_*)\ge (1+\varepsilon)\HH(G).
		\end{multline*}		
		This shows \eqref{eq:measu}.
		\end{proof}
	\begin{proof}[Proof of \eqref{eq:ener}]
		Without loss of generality, we may assume that $\diam(E)=1$. Fix $I\in\mathcal{I_*}$, and let $J_I$ be a child of $I$ satisfying $J_I\in\mathcal{I}$. We claim that we may apply \propref{prop:main prop} with $J=J_I$ and $G=G\cap J_I$. Indeed, assumption (a) is the same as in Corollary \ref{cor:main prop}, and:
		\begin{itemize}
			\item assumption (b) holds since $\HH(J_I)\le \HH(J)\le c_1C_0^{-1} M^{-1}$.
			\item assumption (c) holds because $(J_I)^1=I$ has non-empty intersection with both $G$ and $J\setminus G$, so in particular $I$ strictly contains some $K\in \cB_{\Delta}$. We assumed that there exists $\theta_K\in K^{1}\subset I$ such that $\|\pi_{\theta_K}^\perp\mu\|_\infty \le M$. Since $I\subset 3J_I$, we may take $\theta_0=\theta_K$.
			\item assumption (d) follows from the definition of $\mathcal{I}$ \eqref{eq:Idef}.
			\item assumption (e) holds because any interval $K$ comprising $J_I\setminus G$ contains some dyadic interval $K'\in\cB_{\Delta}$, and since $(K')^1\subset 3K$, we may take $\theta_K \coloneqq \theta_{K'}$.
		\end{itemize}
		We checked all the assumptions of \propref{prop:main prop}, and so we may conclude that
		\begin{multline*}
		\int_{E}\int_0^1 \frac{\mu(X(x,3J_I,r))}{r}\, \frac{dr}{r}d\mu(x)\\
		\le C_{\mathsf{Prop}}\int_{E}\int_0^1 \frac{\mu(X(x,G\cap J_I,r))}{r}\, \frac{dr}{r}d\mu(x) + C_{\mathsf{Prop}}\HH(J_I)\mu(E).
		\end{multline*}
		Summing over $I\in\mathcal{I}_*$ yields
		\begin{align*}
		\int_{E}\int_0^1& \frac{\mu(X(x,G_*,r))}{r}\, \frac{dr}{r}d\HH(x)\\
		&=\sum_{I\in\mathcal{I_*}}\int_{E}\int_0^1 \frac{\mu(X(x,I,r))}{r}\, \frac{dr}{r}d\mu(x)\\
		&\le \sum_{I\in\mathcal{I_*}}\int_{E}\int_0^1 \frac{\mu(X(x,3J_I,r))}{r}\, \frac{dr}{r}d\mu(x)\\
		&\le \sum_{I\in\mathcal{I_*}}C_{\mathsf{Prop}}\int_{E}\int_0^1 \frac{\mu(X(x,G\cap J_I,r))}{r}\, \frac{dr}{r}d\mu(x)+ \sum_{I\in\mathcal{I_*}}C_{\mathsf{Prop}}\HH(J_I)\mu(E)\\
		&\le C_{\mathsf{Prop}}\int_{E}\int_0^1 \frac{\mu(X(x,G,r))}{r}\, \frac{dr}{r}d\mu(x)+ C_{\mathsf{Prop}}\HH(J)\mu(E).
		\end{align*}
		This shows \eqref{eq:ener}.
		\end{proof}
		\begin{proof}[Proof of \eqref{eq:dyda}]
		Let $I\in\cB_{\Delta,*}$, so that
		\begin{equation}\label{eq:max}
		I\cap G_*=\varnothing\quad \text{and}\quad I^1\cap G_*\neq\varnothing.
		\end{equation}
		We want to prove that $I\in\cB_{\Delta}$. Since $G\subset G_*$, it is clear that $I\cap G=\varnothing,$ so we only need to show that
		\begin{equation}\label{eq:goal3}
		I^1\cap G\neq\varnothing.
		\end{equation} 
		
		Let $I'$ be the dyadic sibling of $I$, that is, the unique interval $I'\in\Delta(J)$ such that $I\cup I' = I^1$. It follows from \eqref{eq:max} that $I'\cap G_*\neq \varnothing$.		
		By the definition of $G_*$, there exists $P\in\mathcal{I}_*$ such that $P\cap I'\neq\varnothing$. Hence, we have either $P \subset I'$ or $I'\subsetneq P$. The latter would imply $I^1\subset P$, which is not possible because $I\cap P\subset I\cap G_{*}=\varnothing$. Thus, we have $P \subset I'$.
		
		Let $J_P\in\mathcal{I}$ be such that $P=(J_P)^1$. By the definition of $\mathcal{I}$ \eqref{eq:Idef} we have
		\begin{equation*}
		\HH(J_P\cap G)\ge (1-\varepsilon)\HH(J_P).
		\end{equation*}
		Since $J_P\subset P\subset I'$, it follows that $I'\cap G\neq\varnothing$. In particular the parent $(I')^1=I^1$ satisfies $I^1\cap G\neq\varnothing$. This gives \eqref{eq:goal3}, and concludes the proof of \eqref{eq:dyda}.
	\end{proof}
	This finishes the proof of Corollary \ref{cor:main prop}.
	\subsection{Proof of \thmref{thm:main thm}}
	\subsubsection*{Preliminaries}	
	Recall that $G^\perp = G+1/4$. Let $J_0\subset \TT$ be a dyadic interval with 
	\begin{equation*}
		2^{-1}c_1C_0^{-1} M^{-1}\le \HH(J_0)\le c_1 C_0^{-1} M^{-1}
	\end{equation*}
	and such that
	\begin{equation*}
		\HH(J_0\cap G^\perp)\ge s\HH(J_0).
	\end{equation*}
	It is clear that such interval exists since $\HH(G^\perp) = \HH(G)\ge s$. Using inner regularity of Lebesgue measure, we may find a closed subset $G'\subset G^\perp\cap J_0$ such that
	\begin{equation*}
		\HH(G')\ge \frac{1}{2}\HH(G^\perp\cap J_0)\ge \frac{s}{2}\HH(J_0).
	\end{equation*}

	Let $\varepsilon=\varepsilon(C_0,M)$ be as in \propref{prop:main prop}. We define $\cG\subset\Delta(J_0)$ as the family of maximal dyadic intervals such that for every $I\in \cG$
	\begin{equation*}
		\HH(I\cap G')\ge (1-\varepsilon)\HH(I).
	\end{equation*}
	It follows from Lebesgue differentiation theorem that
	\begin{equation*}
		\HH\bigg(G'\setminus \bigcup_{I\in\cG} I\bigg) = 0.
	\end{equation*}
	In particular,
	\begin{equation*}
		\HH\bigg(\bigcup_{I\in\cG} I\bigg)\ge \HH(G')\ge\frac{s}{2}\HH(J_0).
	\end{equation*}
	Let $\cG_0\subset\cG$ be a finite sub-collection such that
	\begin{equation}\label{eq:meas}
			\HH\bigg(\bigcup_{I\in\cG_0} I\bigg) \ge \frac{1}{2}\HH\bigg(\bigcup_{I\in\cG} I\bigg)\ge \frac{s}{4}\HH(J_0).
	\end{equation}
	Set
	\begin{equation*}
		G_0 = \bigcup_{I\in\cG_0} \overline{I},
	\end{equation*}
	so that $G_0$ is a finite union of closed dyadic intervals.

	Without loss of generality, we may assume that $\diam(E)=1$. For each $I\in\cG_0$ we apply \propref{prop:main prop} (with $J=I$ and $G=G'\cap I$; it is straightforward to see that all the assumptions are satisfied) to conclude that
	\begin{multline}\label{eq:6}
		\int_{E}\int_0^1 \frac{\mu(X(x,\overline{I},r))}{r}\, \frac{dr}{r}d\mu(x)\\
		\le C_{\mathsf{Prop}}\int_{E}\int_0^1 \frac{\mu(X(x,G'\cap I,r))}{r}\, \frac{dr}{r}d\mu(x) + C_{\mathsf{Prop}}\HH(I)\mu(E).
	\end{multline}
	Summing \eqref{eq:6} over $I\in\cG_0$ we get
	\begin{multline}\label{eq:7}
		\int_{E}\int_0^1 \frac{\mu(X(x,G_0,r))}{r}\, \frac{dr}{r}d\mu(x)\\
		\le C_{\mathsf{Prop}}\int_{E}\int_0^1 \frac{\mu(X(x,G',r))}{r}\, \frac{dr}{r}d\mu(x) + C_{\mathsf{Prop}}\HH(G_0)\mu(E).
	\end{multline}
	Notice also that if $\cB_{\Delta,0}$ are maximal dyadic intervals contained in $J_0\setminus G_0$, and $I\in \cB_{\Delta,0}$, then $I^1$ contains some interval from $\cG_0$, and in particular $I^1\cap G'\neq\varnothing$. Since $G'\subset G^\perp$, we get from \eqref{eq:projbdd} that there exists $\theta_I\in I^1$ such that $\|\pi_{\theta_I}\mu\|\le M$. Hence, $G_0$ satisfies all the assumptions of Corollary \ref{cor:main prop}, except perhaps for the measure assumption \eqref{eq:measu}.
	
	\subsubsection*{Iteration} We are in position to start the iteration. Assume for a moment that $\HH(G_0)<(1-\varepsilon)\HH(J_0)$ so that $G_0$ satisfies all the assumptions of Corollary \ref{cor:main prop}. We apply Corollary \ref{cor:main prop}, and we define $G_1\coloneqq (G_0)_*$, so that
	\begin{equation*}
		\HH(G_1)\ge (1+\varepsilon)\HH(G_0) \ge\frac{s(1+\varepsilon)}{4}\HH(J_0),
	\end{equation*}
	and all the other conclusions of Corollary \ref{cor:main prop} hold for $G_1$. If $\HH(G_1)<(1-\varepsilon)\HH(J_0)$, then we may apply Corollary \ref{cor:main prop} yet again to get a set $G_2\coloneqq (G_1)_*$. 
	
	In general, if after $k$-applications of Corollary \ref{cor:main prop} we get a set $G_{k}\coloneqq (G_{k-1})_*$ satisfying $\HH(G_{k})<(1-\varepsilon)\HH(J_0)$, then we may continue applying Corollary \ref{cor:main prop}. If for some $k=k_0$ we get $\HH(G_{k_0})\ge(1-\varepsilon)\HH(J_0)$, then we may apply \propref{prop:main prop} instead (with $G=G_{k_0}, J=J_0$), so that
	\begin{multline*}
	\int_{E}\int_0^1 \frac{\mu(X(x,3J_0,r))}{r}\, \frac{dr}{r}d\mu(x)\\
	\le C_{\mathsf{Prop}}\int_{E}\int_0^1 \frac{\mu(X(x,G_{k_0},r))}{r}\, \frac{dr}{r}d\mu(x) + C_{\mathsf{Prop}}\HH(J_0)\mu(E).
	\end{multline*}
	Recall that for each $k$ we had $G_{k+1}=(G_k)_*$, so that by \eqref{eq:ener}
	\begin{multline*}
		\int_{E}\int_0^1 \frac{\mu(X(x,G_{k+1},r))}{r}\, \frac{dr}{r}d\mu(x)\\
		\le C_{\mathsf{Prop}}\int_{E}\int_0^1 \frac{\mu(X(x,G_{k},r))}{r}\, \frac{dr}{r}d\mu(x) + C_{\mathsf{Prop}}\HH(J_0)\mu(E).
	\end{multline*}
	Putting the two estimates above together (the second one used $k_0$ times), and also recalling \eqref{eq:7}, we get
	\begin{multline}\label{eq:8}
		\int_{E}\int_0^1 \frac{\mu(X(x,3J_0,r))}{r}\, \frac{dr}{r}d\mu(x)\\
		\le C_{\mathsf{Prop}}^{k_0+1}\int_{E}\int_0^1 \frac{\mu(X(x,G_{0},r))}{r}\, \frac{dr}{r}d\mu(x) + (k_0+1)C_{\mathsf{Prop}}^{k_0+1}\HH(J_0)\mu(E)\\ 
		\le C_{\mathsf{Prop}}^{k_0+2}\int_{E}\int_0^1 \frac{\mu(X(x,G',r))}{r}\, \frac{dr}{r}d\mu(x) + (k_0+2)C_{\mathsf{Prop}}^{k_0+2}\HH(J_0)\mu(E).
	\end{multline}
	
	\subsubsection*{Bounding the number of iterations} We claim that the iteration ends (i.e. we obtain a set $G_{k_0}$ with $\HH(G_{k_0})\ge(1-\varepsilon)\HH(J_0)$) after at most
	\begin{equation}\label{eq:9}
		k_0\lesssim_{s,\varepsilon} 1
	\end{equation}
 	steps.
	Indeed, we had
	\begin{equation*}
		\HH(G_0) = \HH\bigg(\bigcup_{I\in\cG_0}I\bigg) \overset{\eqref{eq:meas}}{\ge}\frac{s}{4}\HH(J_0),
	\end{equation*}
	and so by \eqref{eq:measu} for each $G_k$ we have a lower bound
		\begin{equation*}
		\HH(G_k) \ge (1+\varepsilon)\HH(G_{k-1})\ge (1+\varepsilon)^k\HH(G_0)\ge \frac{s(1+\varepsilon)^k}{4}\HH(J_0).
	\end{equation*}
	Taking $k_0=k_0(s,\varepsilon)$ so large that $s(1+\varepsilon)^{k_0}/4\ge (1-\varepsilon)$, we see that the iterative procedure described above ends after at most $k_0$ applications of Corollary \ref{cor:main prop}.
	
	\subsubsection*{End of the proof}
	Taking into account estimates \eqref{eq:8} and \eqref{eq:9}, the fact that $\varepsilon=\varepsilon(M,C_0),$ $C_{\mathsf{Prop}}=C_{\mathsf{Prop}}(M, C_0),$ $\HH(J)\le 1$, and that $G'\subset G^\perp$, we get
	\begin{multline*}
		\int_{E}\int_0^1 \frac{\mu(X(x,3J_0,r))}{r}\, \frac{dr}{r}d\mu(x)\\
		\le C(M,C_0,s)\int_{E}\int_0^1 \frac{\mu(X(x,G^\perp,r))}{r}\, \frac{dr}{r}d\mu(x) + C(M,C_0,s)\mu(E).
	\end{multline*}
	Hence, by Corollary \ref{cor:changtolsa}
	\begin{equation*}
	\int_{E}\int_0^1 \frac{\mu(X(x,3J_0,r))}{r}\, \frac{dr}{r}d\mu(x)\lesssim_{M,C_0,s} \mu(E).
	\end{equation*}
	Let $M_0=M_0(M,C_0,s)$ be a big constant. We define
	\begin{equation*}
	E_*\coloneqq \big\{x\in E\ :\ \int_0^1 \frac{\mu(X(x,3J_0,r))}{r}\, \frac{dr}{r}\le M_0\big\}.
	\end{equation*}
	By Chebyshev's inequality, if $M_0$ is chosen big enough, we have
	\begin{equation*}
	\mu(E_*)\ge\frac{\mu(E)}{2}.
	\end{equation*}
	Applying \propref{prop:conical energy} to $E_*$ and $3J_0$, and recalling that $\HH(J_0)\sim C_0^{-1}M^{-1}$, we obtain a Lipschitz graph $\Gamma$ with $\lip(\Gamma)\lesssim_{M,C_0} 1$ and 
	\begin{equation*}
	\HH(\Gamma\cap E_*)\gtrsim_{C_0,M,M_0} \mu(E).
	\end{equation*}
	This finishes the proof of \thmref{thm:main thm}. 
	\begin{flushright}
		$\square$
	\end{flushright}

	The remainder of the paper is dedicated to the proof of \propref{prop:main prop}.
	
	\section{Rectangles and generalized cubes}\label{sec:rectangles}
	
	Suppose that $E\subset \R^2$ is a bounded AD-regular set with constant $C_0$, and set $\mu = \HH|_{E}.$ Since \propref{prop:main prop} is scale-invariant, we may assume without loss of generality that $\diam(E)=1$.
	
	Let $J,G\subset\TT$ be as in \propref{prop:main prop}. By rotating $E$, we may assume that $J$ is centered at $1/4$, so that the cone $X(0,J)$ is centered on the vertical axis. Note that that $\pi_0=\pi_{1/4}^\perp$ is the projection to the horizontal axis, i.e., $\pi_{0}(x,y)=x$. Recall that there exists $\theta_0\in 3J$ such that
	\begin{equation}\label{eq:bddprojtheta0}
		\|\pi_{\theta_0}^\perp\mu\|_\infty \le M.
	\end{equation}	
	
	\subsection{Rectangles}
	Throughout the article we will be working with many rectangles, typically with one side much longer than the other. Let us fix some notation.
	
	Given a rectangle $\cR\subset\R^2$, we will denote the length of its shorter side by $\ell(\cR)$, and the length of its longer side by $\ELL(\cR)$. We will also write $\theta(\cR)\in[0,1/2)\subset\TT$ to denote the ``direction'' of $\cR$, so that $\ell_{\theta(\cR)}=\spn((\cos(2\pi\theta(\cR)),\sin(2\pi\theta(\cR))))$ is parallel to the longer sides of $\cR$ (for squares, it doesn't matter which of the two directions we choose).
	
	Given a constant $C>0$ and a rectangle $\cR$, we will sometimes write $C\cR$ to denote the (unique) rectangle with the same center as $\cR$, $\ell(C\cR)=C\ell(\cR),\, \ELL(C\cR)=C\ELL(\cR)$, and such that their longer sides are parallel to each other.
	
	Most of the rectangles $\cR$ we will be working with will have a fixed direction $\theta(\cR) = 1/4$, and a fixed aspect ratio $\ELL(\cR)/\ell(\cR)=\HH(J)^{-1}$. In other words, they will be very tall, vertically aligned rectangles. We fix notation specific to these rectangles.
	
	Given $x\in\R^2$ and $r>0$ we set
	\begin{equation*}
	\cR(x,r) = x + \bigg[-\frac{r}{2},\, \frac{r}{2}\bigg]\times \bigg[-\frac{r}{2\HH(J)},\, \frac{r}{2\HH(J)}\bigg],
	\end{equation*}
	so that $\ell(\cR(x,r))=r$ and $\ELL(\cR(x,r))=\HH(J)^{-1}r$.
	Note that $\pi_0(\cR(x,r)) = \pi_0(x)+[-r/2, r/2]$. 
	
	\begin{lemma}\label{lem:ADRrectangles}	
		Let $\cR$ be a rectangle, and suppose that for some $\theta\in\TT$ with
		\begin{equation}\label{eq:theta close}
		|\theta - \theta(\cR)|\lesssim \frac{\ell(\cR)}{\ELL(\cR)}
		\end{equation}
		we have $\|\pi_\theta^\perp \mu\|_{L^\infty}\le M$. Then,
		\begin{equation}\label{eq:ADRrectangles2}
			\mu(\cR)\lesssim M\ell(\cR).
		\end{equation}
	\end{lemma}
	\begin{proof}
		Let $\cR$ and $\theta$ be as above, and set $\alpha = |\theta - \theta(\cR)|\cdot2\pi$. It follows from elementary trigonometry that
		\begin{equation*}
		\HH(\pi_\theta^\perp(\cR)) = \ell(\cR)\bigg(\cos(\alpha)+\frac{\ELL(\cR)}{\ell(\cR)}\sin(\alpha)\bigg).
		\end{equation*}
		From \eqref{eq:theta close} we have $\alpha\lesssim \frac{\ell(\cR)}{\ELL(\cR)}$, and so
		\begin{equation*}
		\HH(\pi_\theta^\perp(\cR))\lesssim \ell(\cR).
		\end{equation*}
		Since $\|\pi_\theta^\perp \mu\|_{L^\infty}\le M$, we get
		\begin{equation*}
		\mu(\cR)\le\mu((\pi_{\theta}^\perp)^{-1}(\pi_\theta^\perp(\cR)))\le M \HH(\pi_\theta^\perp(\cR))\lesssim M\ell(\cR).
		\end{equation*}
	\end{proof}
	\begin{cor}
		For any $x\in\R^2$ and $r>0$ we have
		\begin{equation}\label{eq:ADRrectangles}
		\mu(\cR(x,r))\lesssim Mr.
		\end{equation}
	\end{cor}
	\begin{proof}
		Observe that for $\cR=\cR(x,r)$ we have $\theta(\cR)=1/4\in J$. Recall that there exists $\theta_0\in 3J$ such that $\|\pi_{\theta_0}^\perp\mu\|_\infty \le M$. Since $|\theta_0-\theta(\cR)|\le 2\HH(J) = 2\ell(\cR)/\ELL(\cR)$, we get from \eqref{eq:ADRrectangles2}
		\begin{equation*}
		\mu(\cR(x,r))\lesssim M\ell(\cR(x,r))= Mr.
		\end{equation*}
	\end{proof}
	
	\subsection{Generalized dyadic cubes}
	We say that a metric space $(X,d)$ has a finite doubling property if any ball $B_X(x,2r)\subset X$ can be covered by finitely many balls of the form $B_X(x_i,r)$. The following is a special case of Theorem 2.1 from \cite{kaenmaki2012existence}.
	\begin{theorem}[\cite{kaenmaki2012existence}]\label{thm:dyadic cubes}
		Let $\rho=1/1000$. Suppose that $(X,d)$ is a metric space with the finite doubling property. Then, for every $k\in\Z$ there exists a collection $\DD_k$ of generalized cubes on $X$ such that the following hold:
		\begin{enumerate}
			\item For each $k\in\Z$, $X=\bigcup_{Q\in\DD_k}Q$, and the union is disjoint.
			\item If $Q_1,Q_2\in\bigcup_k\DD_k$ satisfy $Q_1\cap Q_2\neq\varnothing$, then either $Q_1\subset Q_2$ or $Q_2\subset Q_1$.
			\item For every $Q\in\DD_k$ there exists $x_Q\in Q$ such that
			\begin{equation*}
			B_X(x_Q, 0.4\rho^k)\subset Q\subset B_X(x_Q,2\rho^k).
			\end{equation*}
		\end{enumerate}
	\end{theorem}
	
	Consider $X=E$ endowed with the metric
	\begin{equation}\label{eq:metric}
	d((x_1,y_1),(x_2,y_2)) = \max\big(|x_1-x_2|,\,\HH(J)\,{|y_1-y_2|}\big).
	\end{equation}
	Note that for $x\in E$ and $r>0$, the ball with respect to $d$ is of the form $B_X(x,r)=\cR(x,2r)\cap E$. 
	
	It is clear that $(E,d)$ has the finite doubling property, and so we may use \thmref{thm:dyadic cubes} to obtain a lattice of generalized cubes $\DD=\bigcup_{k\in\Z}\DD_k$ associated to $(E,d)$. 
	
	Given $Q\in\DD_k$, we will write
	\begin{align*}
	\ell(Q) &\coloneqq 4\rho^k,\\
	\Ch(Q) &\coloneqq \{P\in\DD_{k+1}\ :\ P\subset Q\},\\
	\DD(Q)&\coloneqq\{P\in\DD\ :\  P\subset Q,\ \ell(P)\le\ell(Q)\}.	
	\end{align*}

	Observe that $Q\subset \cR(x_Q,\ell(Q))\cap E$.
	We set 
	\begin{align}	
	\cR_Q&\coloneqq \cR(x_Q,\ell(Q)),\label{eq:Rq def}\\
	\ELL(Q) &\coloneqq\HH(J)^{-1}\ell(Q),\notag
	\end{align}
	so that $\ell(\cR_Q)=\ell(Q)$ and $\ELL(\cR_Q)=\ELL(Q)$.
	
	Note that if $P, Q\in\DD$ satisfy $P\cap Q=\varnothing$ and $\ell(P)\ge\ell(Q)$, then by (3) in \thmref{thm:dyadic cubes} we have $d(x_P, x_Q)\ge 0.1\ell(P)\ge 0.05\ell(P)+0.05\ell(Q)$, so in particular $0.1\cR_P\cap 0.1\cR_Q = \varnothing.$ We set
	\begin{equation*}
	\cR(Q)\coloneqq 0.1\cR_Q.
	\end{equation*}
	We record for future reference that
	\begin{align*}
	\cR(Q)\cap E\subset &Q\subset \cR_Q\cap E,\\
	2\cR_Q&\subset 2\cR_P\quad\quad\quad\text{if $Q\subset P$,}\\
	\cR(Q)\cap&\cR(P)=\varnothing\quad\quad\text{if $Q\cap P=\varnothing$}.
	\end{align*}

	Observe also that for any $C>0$ such that $C\ell(Q)\lesssim \diam(E)=1$ we have
	\begin{equation}\label{eq:ADR cubes}
	CC_0\,\ell(Q)\lesssim \mu(C\cR_Q)\overset{\eqref{eq:ADRrectangles}}{\lesssim} CM\ell(Q).
	\end{equation}
	In particular,
	\begin{equation*}
	C_0\,\ell(Q)\lesssim \mu(Q)\lesssim M\ell(Q).
	\end{equation*}
	
	\section{Conical energies}\label{sec:conical}
	Let $A=A(C_0,M)\ge 1000$ be a large constant which we will fix later on. Inspired by \cite{chang2017analytic} and \cite{dabrowski2020cones}, we introduce the following conical energy associated to the set of directions $G\subset J$.
		For any $Q\in\DD$ we set
		\begin{equation}\label{eq:energydef}
		\E_G(Q) \coloneqq\frac{1}{\mu(Q)} \int_{2A\cR_{Q}}\int_{A^{-1}\ELL(Q)}^{A^3\ELL(Q)} \frac{\mu(X(x,G,r))}{r}\, \frac{dr}{r}d\mu(x).
		\end{equation}	
	We have the following easy upper bound for $\E_G(Q)$.
	\begin{lemma}
		For any $Q\in\DD$ we have
		\begin{equation}\label{eq:trivial energy bound}
		\E_G(Q)\lesssim_{A,M,C_0} \HH(J).
		\end{equation}
	\end{lemma}
	\begin{proof}
		Observe that for any $x\in 2A\cR_Q$ and $r\in (A^{-1}\ELL(Q), A^3\ELL(Q))$ we have
		\begin{equation*}
		X(x,G,r)\subset X(x,J,A^3\ELL(Q))\subset \cR(x,A^4\ell(Q)),
		\end{equation*}
		so that
		\begin{equation*}
		\mu(X(x,G,r))\le\mu(\cR(x,A^4\ell(Q)))\overset{\eqref{eq:ADRrectangles}}{\lesssim}A^4M\,\ell(Q).
		\end{equation*}
		Hence,
		\begin{multline*}
		\E_G(Q) =\frac{1}{\mu(Q)} \int_{2A\cR_{Q}}\int_{A^{-1}\ELL(Q)}^{A^3\ELL(Q)} \frac{\mu(X(x,G,r))}{r}\, \frac{dr}{r}d\mu(x)\\
		 \lesssim_{A,M} \frac{1}{\mu(Q)} \int_{2A\cR_{Q}}\int_{A^{-1}\ELL(Q)}^{A^3\ELL(Q)} \frac{\ell(Q)}{\ELL(Q)}\, \frac{dr}{r}d\mu(x)\\
		  \sim_A \HH(J)\frac{\mu(2A\cR_Q)}{\mu(Q)}\overset{\eqref{eq:ADR cubes}}{\lesssim}_{A,M,C_0} \HH(J).
		\end{multline*}
	\end{proof}
	
	\subsection{Stopping time argument}	
	Given a small constant $\delta=\delta(A,M,C_0)>0$, we consider the following stopping time condition. For $R\in\DD$, we define the family $\BCE(R)$ as the family of maximal cubes $Q\in\DD(R)$ such that
	\begin{equation}\label{eq:BCE def}
	\sum_{S\in\DD: Q\subset S\subset R} \E_G(S)\ge \delta\HH(J).
	\end{equation}
	We define also $\Tree(R)$ as the subfamily of $\DD(R)$ consisting of cubes that are not strictly contained in any cube from $\BCE(R)$. Note that it may happen that $R\in\BCE(R)$, in which case $\Tree(R)=\{R\}$.
	
	\begin{lemma}
		For any $R\in\DD$ we have
		\begin{equation}\label{eq:energy small inside tree}
		\sum_{Q\in\Tree(R)\setminus\BCE(R)}\E_G(Q)\mu(Q)\le \delta\HH(J)\mu(R),
		\end{equation}
		and
		\begin{equation}\label{eq:energy lower bound}
		\delta\HH(J)\sum_{P\in\BCE(R)}\mu(P)\le \sum_{Q\in\Tree(R)}\E_G(Q)\mu(Q)\lesssim_{A,M,C_0} \HH(J)\mu(R).
		\end{equation}
	\end{lemma}
	\begin{proof}
		
		We start by proving \eqref{eq:energy small inside tree}. Observe that
		\begin{multline*}
		\sum_{Q\in\Tree(R)\setminus\BCE(R)}\E_G(Q)\mu(Q)
		 = \sum_{Q\in\Tree(R)\setminus\BCE(R)}\int \E_G(Q)\one_{Q}(x)\,d\mu(x)\\
		 		= \int \sum_{Q\in\Tree(R)\setminus\BCE(R)} \E_G(Q)\one_{Q}(x)\,d\mu(x).
		\end{multline*}
		Let $x\in R$, and let $P\in\Tree(R)\setminus\BCE(R)$ be a cube with $x\in P$. Recalling that $P\notin\BCE(R)$ and the definition of $\BCE(R)$ \eqref{eq:BCE def}, we get
		\begin{equation*}
		\sum_{P\subset Q\subset R}\E_G(Q) <\delta\HH(J).
		\end{equation*}
		Since $P$ was an arbitrary cube with $P\in\Tree(R)\setminus\BCE(R)$ and $x\in P$, this gives
		\begin{equation*}
		\sum_{Q\in\Tree(R)\setminus\BCE(R)} \E_G(Q)\one_{Q}(x) \le \delta\HH(J).
		\end{equation*}
		Integrating over $x\in R$ yields
		\begin{equation*}
		\sum_{Q\in\Tree(R)\setminus\BCE(R)}\E_G(Q)\mu(Q)\le \delta\HH(J)\mu(R).
		\end{equation*}
		This proves \eqref{eq:energy small inside tree}. 
		
		The upper bound in \eqref{eq:energy lower bound} follows from \eqref{eq:energy small inside tree} and the trivial bound \eqref{eq:trivial energy bound} applied to $Q\in\BCE(R)$:
		\begin{equation*}
		\sum_{Q\in\BCE(R)}\E_G(Q)\mu(Q)\lesssim_{A,M,C_0} \HH(J)\sum_{Q\in\BCE(R)}\mu(Q)\le\HH(J)\mu(R).
		\end{equation*}
		
		Now we prove the lower bound in \eqref{eq:energy lower bound}. 
		We have
		\begin{multline}\label{eq:est1}
		\sum_{Q\in\Tree(R)}\E_G(Q)\mu(Q)
		=\int \sum_{Q\in\Tree(R)}\E_G(Q)\one_{Q}(x)\,d\mu(x)\\
		\ge \int\sum_{P\in\BCE(R)}\sum_{Q\in\Tree(R),\,P\subset Q}\E_G(Q)\one_{P}(x)\,d\mu(x).
		\end{multline}
		By \eqref{eq:BCE def} we have for every $P\in\BCE(R)$
		\begin{equation*}
		\sum_{Q\in\Tree(R),\,P\subset Q}\E_G(Q)\ge\delta\HH(J).
		\end{equation*}
		Hence,
		\begin{multline*}
		\int\sum_{P\in\BCE(R)}\sum_{Q\in\Tree(R),\,P\subset Q}\E_G(Q)\one_{P}(x)\,d\mu(x)\ge \delta\HH(J)\sum_{P\in\BCE(R)}\int\one_{P}(x)\,d\mu(x)\\
		= \delta\HH(J)\sum_{P\in\BCE(R)}\mu(P).
		\end{multline*}
		Together with \eqref{eq:est1}, this gives the desired estimate.
	\end{proof}
	
	\subsection{Corona decomposition}\label{subsec:corona}
	We are ready to perform the corona decomposition. Let $k(J)\in\Z$ be the largest integer such that for $Q\in\DD_{k(J)}$ we have 
	\begin{equation*}
	\ELL(Q) = 4\HH(J)^{-1}\rho^{k(J)}\ge 1.
	\end{equation*}
	Set $\DD_*=\bigcup_{k\ge k(J)}\DD_k$, and
	\begin{equation*}
	\Top_0 = \{\DD_{k(J)}\}.
	\end{equation*}
	If $\Top_k$ has already been defined, we set
	\begin{equation*}
	\Top_{k+1} = \bigcup_{R\in\Top_k}\bigcup_{Q\in\BCE(R)}\Ch(Q).
	\end{equation*}
	Finally,
	\begin{equation*}
	\Top = \bigcup_{k\ge 0 }\Top_k.
	\end{equation*}	
	Observe that
	\begin{equation*}
	\bigcup_{R\in\Top}\Tree(R) = \DD_*.
	\end{equation*}
	The following is a fairly standard computation.
	\begin{lemma}
		We have
		\begin{multline}\label{eq:energytocubes}
		\HH(J)\mu(E)+\int_{E}\int_0^1 \frac{\mu(X(x,G,r))}{r}\, \frac{dr}{r}d\mu(x)\\
		\sim_{A,M} \HH(J)\mu(E)+\sum_{Q\in\DD_*}\E_G(Q)\mu(Q).
		\end{multline}
	\end{lemma}
	\begin{proof}
		Fix $k\ge k(J)$. Using the fact that for $Q\in \DD_k$ the rectangles $2A\cR_Q$ have only bounded overlaps (with bound depending on $A$), we have
		\begin{equation*}
		\sum_{Q\in\DD_k}\E_G(Q)\mu(Q) \sim_A \int_{E}\int_{4A^{-1}\HH(J)^{-1}\rho^k}^{4A^3\HH(J)^{-1}\rho^k} \frac{\mu(X(x,G,r))}{r}\, \frac{dr}{r}d\mu(x).
		\end{equation*}
		Summing over $k\ge k(J)$ we get
		\begin{equation*}
		\sum_{Q\in\DD_*}\E_G(Q)\mu(Q) \sim_A \int_{E}\int_0^{4A^3\HH(J)^{-1}\rho^{k(J)}} \frac{\mu(X(x,G,r))}{r}\, \frac{dr}{r}d\mu(x).
		\end{equation*}
		Recalling that $1\le 4\HH(J)^{-1}\rho^{k(J)}\lesssim 1$, we get that
		\begin{equation*}
		\sum_{Q\in\DD_*}\E_G(Q)\mu(Q) \sim_A \int_{E}\int_0^{CA^3} \frac{\mu(X(x,G,r))}{r}\, \frac{dr}{r}d\mu(x)
		\end{equation*}
		for some constant $1\le C\lesssim 1$. This is obviously no-smaller than the integral on the left hand side of \eqref{eq:energytocubes}. 
		
		To see the converse estimate, note that for $r>1$ we have $X(x,G,r)\cap E\subset \cR(x,2\HH(J))$, so that
		\begin{multline*}
		\int_{E}\int_1^{CA^3} \frac{\mu(X(x,G,r))}{r}\, \frac{dr}{r}d\mu(x) \lesssim \int_{E}\int_1^{CA^3} \frac{\mu(\cR(x,2\HH(J)))}{r}\, \frac{dr}{r}d\mu(x)\\
		\overset{\eqref{eq:ADRrectangles}}{\lesssim} M\HH(J)\int_{E}\int_1^{CA^3} \frac{1}{r^2}\, dr d\mu(x)\lesssim M\HH(J)\mu(E).
		\end{multline*}
	\end{proof}

	The family $\Top$ satisfies the following packing condition.
	\begin{lemma}
		We have
		\begin{equation}\label{eq:packing-for-Top}
		\sum_{R\in\Top}\mu(R)\lesssim_{\delta,A} (\HH(J))^{-1}\int_{E}\int_0^1 \frac{\mu(X(x,G,r))}{r}\, \frac{dr}{r}d\mu(x) + \mu(E).
		\end{equation}
	\end{lemma}
	\begin{proof}
		First, we use the fact that the cubes $R\in\Top_0$ are pairwise disjoint to estimate
		\begin{equation*}
		\sum_{R\in\Top_0}\mu(R)\le \mu(E).
		\end{equation*}
		This gives the second term on the right hand side of \eqref{eq:packing-for-Top}. 
		
		Moving on to $\Top\setminus\Top_0$, we compute
%
		\begin{multline*}
		\sum_{R\in\Top\setminus\Top_0}\mu(R) = \sum_{k\ge 0} \sum_{R\in\Top_{k+1}}\mu(R)
		= \sum_{k\ge 0} \sum_{R\in\Top_{k}}\sum_{Q\in\BCE(R)}\sum_{P\in\Ch(Q)}\mu(P)\\
		=\sum_{k\ge 0} \sum_{R\in\Top_{k}}\sum_{Q\in\BCE(R)}\mu(Q)
		 \overset{\eqref{eq:energy lower bound}}{\le} (\delta\HH(J))^{-1} \sum_{k\ge 0} \sum_{R\in\Top_{k}}\sum_{Q\in\Tree(R)}\E_G(Q)\mu(Q)
		 \\
		 = (\delta\HH(J))^{-1} \sum_{Q\in\DD_*}\E_G(Q)\mu(Q)
		 \\
		 \lesssim_{A,M} (\delta\HH(J))^{-1}\int_{E}\int_0^1 \frac{\mu(X(x,G,r))}{r}\, \frac{dr}{r}d\mu(x) + \delta^{-1}\mu(E).
		\end{multline*}
	\end{proof}
	
	Consider the following conical energy associated to $3J$:
	\begin{equation*}
	\E_J(Q) \coloneqq\frac{1}{\mu(Q)} \int_{Q}\int_{\rho\ELL(Q)}^{\ELL(Q)} \frac{\mu(X(x,3J,r))}{r}\, \frac{dr}{r}d\mu(x).
	\end{equation*}
	Arguing as in \eqref{eq:energytocubes}, it is easy to show that
	\begin{equation}\label{eq:est3}
	\int_{E}\int_0^1 \frac{\mu(X(x,3J,r))}{r}\, \frac{dr}{r}d\mu(x) \lesssim \sum_{Q\in\DD_*}\E_J(Q)\mu(Q).
	\end{equation}
	
	We divide the conical energy $\E_J(Q)$ into an ``interior'' and ``exterior'' part, which will be dealt with separately:
	\begin{align*}
	\E^{{int}}_J(Q) &\coloneqq \frac{1}{\mu(Q)} \int_{Q}\int_{\rho\ELL(Q)}^{\ELL(Q)} \frac{\mu(X(x,0.5J,r))}{r}\, \frac{dr}{r}d\mu(x),\\
	\E^{{ext}}_J(Q) &\coloneqq \frac{1}{\mu(Q)} \int_{Q}\int_{\rho\ELL(Q)}^{\ELL(Q)} \frac{\mu(X(x,3J\setminus 0.5J,r))}{r}\, \frac{dr}{r}d\mu(x).
	\end{align*}
	
	We define also the following modification of $\E_J^{ext}(Q)$
	\begin{equation*}
	\widetilde{\E}_J^{ext}(Q) \coloneqq\frac{1}{\mu(Q)} \int_{Q} \frac{\mu(X(x,3J\setminus 0.5J,\rho\ELL(Q), \ELL(Q)))}{\ELL(Q)}\, d\mu(x).
	\end{equation*}
	\begin{lemma}
		We have
		\begin{equation}\label{eq:est4}
			\sum_{Q\in\DD_*} \E_J^{ext}(Q)\mu(Q)\lesssim \sum_{Q\in\DD_*} \EE^{ext}(Q)\mu(Q).
		\end{equation}
	\end{lemma}
	\begin{proof}
		Given $x\in Q$, we set
		\begin{equation*}
			X(x,Q) = X(x,3J\setminus 0.5J,\rho\ELL(Q), \ELL(Q)).
		\end{equation*}
		If $Q=Q_0(x)\supset Q_1(x)\supset Q_2(x)\supset\dots$ is a sequence of cubes such that for all $i\in\mathbb{N}$ we have $Q_{i+1}(x)\in\Ch(Q_i(x))$ and $x\in Q_i(x)$, then
		\begin{equation*}
			\mu(X(x,3J\setminus 0.5 J,\ELL(Q))) = \sum_{i\in\mathbb{N}} \mu(X(x,Q_i(x))).
		\end{equation*}
	Thus, for $x\in Q$ and $\rho\ELL(Q)<r<\ELL(Q)$
	\begin{equation*}
		\frac{\mu(X(x,3J\setminus 0.5 J,r)}{r} \lesssim \sum_{i\in\mathbb{N}} \frac{\mu(X(x,Q_i(x)))}{\ELL(Q)} = \sum_{i\in\mathbb{N}} \frac{\mu(X(x,Q_i(x)))}{\ELL(Q_i(x))}\cdot\frac{\ell(Q_i(x))}{\ell(Q)}.
	\end{equation*}
	Integrating over $x\in Q$ and $\rho\ELL(Q)<r<\ELL(Q)$ yields
	\begin{equation*}
		\E_J^{ext}(Q)\mu(Q) \lesssim \sum_{P\in\DD(Q)} \EE^{ext}(P)\mu(P)\frac{\ell(P)}{\ell(Q)}.
	\end{equation*}
	We sum over $Q\in\DD_*$ and conclude that
	\begin{multline*}
		\sum_{Q\in\DD_*}\E_J^{ext}(Q)\mu(Q) \lesssim \sum_{Q\in\DD_*}\sum_{P\in\DD(Q)} \EE^{ext}(P)\mu(P)\frac{\ell(P)}{\ell(Q)}\\
		 = \sum_{P\in\DD_*}\EE^{ext}(P)\mu(P)\sum_{Q\in\DD_*,\, Q\supset P}\frac{\ell(P)}{\ell(Q)} \lesssim \sum_{P\in\DD_*}\EE^{ext}(P)\mu(P),
	\end{multline*}
	where in the last inequality we used the fact that the inner sum was a geometric series.
	\end{proof}

	We will prove the following estimates for the interior and exterior energies.
	\begin{lemma}\label{lem:energy estimates}
		If $\varepsilon=\varepsilon(M,C_0)$ is chosen small enough, then for any $R\in\Top$ we have
		\begin{equation}\label{eq:interior energy est}
			\sum_{Q\in\Tree(R)}\E_J^{int}(Q)\mu(Q)
			\lesssim_{C_0} \sum_{Q\in\Tree(R)}\E_G(Q)\mu(Q).
		\end{equation}
		Furthermore, if $A=A(C_0,M)$ is chosen big enough, and $\delta=\delta(A,M,C_0)$ is chosen small enough, then
		\begin{equation}\label{eq:exterior energy est}
			\sum_{Q\in\Tree(R)}\EE^{ext}(Q)\mu(Q) \lesssim_{C_0,M} \HH(J)\mu(R).
		\end{equation}
	\end{lemma}
	We prove \eqref{eq:interior energy est} in Section \ref{sec:interior energy}, and \eqref{eq:exterior energy est} in Section \ref{sec:exterior energy}. Now we show how \propref{prop:main prop} follows from the estimates above.
	\begin{proof}[Proof of \propref{prop:main prop}]
		Recall that our goal is to prove
		\begin{multline}\label{eq:goal}
			\int_{E}\int_0^1 \frac{\mu(X(x,3J,r))}{r}\, \frac{dr}{r}d\mu(x)\\
			\lesssim_{C_0,M}\int_{E}\int_0^1 \frac{\mu(X(x,G,r))}{r}\, \frac{dr}{r}d\mu(x) + \HH(J)\mu(E).
		\end{multline}
	By \eqref{eq:est3}, the left hand side is bounded by
	\begin{multline*}
		\sum_{Q\in\DD_*}\E_J(Q)\mu(Q) = \sum_{Q\in\DD_*}\E_J^{int}(Q)\mu(Q) + \sum_{Q\in\DD_*}\E_J^{ext}(Q)\mu(Q)\\
		\overset{\eqref{eq:est4}}{\lesssim} \sum_{Q\in\DD_*}\E_J^{int}(Q)\mu(Q) + \sum_{Q\in\DD_*}\EE^{ext}(Q)\mu(Q)\\
		 = \sum_{R\in\Top}\sum_{Q\in\Tree(R)}\E_J^{int}(Q)\mu(Q) + \sum_{R\in\Top}\sum_{Q\in\Tree(R)}\EE^{ext}(Q)\mu(Q) \eqqcolon S_1 + S_2.
	\end{multline*}
	To estimate $S_1$, we apply \eqref{eq:interior energy est} and \eqref{eq:energytocubes} to conclude
	\begin{equation*}
		S_1\lesssim \sum_{R\in\Top}\sum_{Q\in\Tree(R)}\E_G(Q)\mu(Q) \lesssim_{A,M} \int_{E}\int_0^1\frac{\mu(X(x,G,r))}{r}\, \frac{dr}{r}d\mu(x) + \HH(J)\mu(E).
	\end{equation*}	
	Regarding $S_2$, using \eqref{eq:exterior energy est} and \eqref{eq:packing-for-Top} yields
	\begin{equation*}
		S_2 \lesssim_{M} \sum_{R\in\Top}\HH(J)\mu(R) \lesssim_{A,\delta} \int_{E}\int_0^1 \frac{\mu(X(x,G,r))}{r}\, \frac{dr}{r}d\mu(x) + \HH(J)\mu(E).
	\end{equation*}
	Recalling that $\delta=\delta(A,M,C_0)$ and $A=A(C_0,M)$, this gives \eqref{eq:goal}.
	\end{proof}
	
	\section{Estimating interior energy and obtaining good cones}\label{sec:interior energy}
	\subsection{Interior energy estimates}
	Recall that in \propref{prop:main prop}, assumption (e), we assumed that $G$ is closed, and that for every interval $I$ which is a connected component of $J\setminus G$ there exists $\theta_I\in 3I$ such that $\|\pi_{\theta_I}^\perp\mu\|_\infty \le M$. We use this property in the following lemma, which is the first step in estimating $\E_J^{int}(Q)$.
	\begin{lemma}\label{lem:littlemeas}
		For any $x\in\R^2$ and $0<r<\infty$ we have
		\begin{equation*}
		\mu(X(x,J\setminus G,r))\lesssim M \HH(J\setminus G)\, r.
		\end{equation*}
		In particular, since $\HH(J\setminus G)\le\varepsilon\HH(J)$, we have
		\begin{equation}\label{eq:littlemeas}
		\mu(X(x,J\setminus G,r))\lesssim M\varepsilon\HH(J)\, r.
		\end{equation}
	\end{lemma}
	\begin{proof}		
		Let $\cB$ denote the intervals comprising $J\setminus G$,  so that for every $I\in\cB$ there exists $\theta_I\in 3I$ such that $\|\pi_{\theta_I}^\perp\mu\|_\infty \le M$. Clearly,
		\begin{equation*}
		X(x,J\setminus G,r)= \bigcup_{I\in\cB} X(x,I,r).
		\end{equation*}
		Observe that each truncated cone $X(x,I,r)$ is contained in some rectangle $\cR_I$ which is centered at $x$, its direction $\theta(\cR_I)\in\TT$ coincides with the midpoint of $I$, and it satisfies $\ell(\cR_I)\sim \HH(I)\,r,\, \ELL(\cR_I)\sim r$. Since 
		\begin{equation*}
		|\theta(R_I) - \theta_I|\le 2\HH(I) \sim \frac{\ell(\cR_I)}{\ELL(\cR_I)},
		\end{equation*}
		we may use \lemref{lem:ADRrectangles} (recall that $\|\pi_{\theta_I}^\perp\mu\|_\infty \le M$) to conclude that
		\begin{equation*}
		\mu(\cR_I)\lesssim M\ell(\cR_I) \sim M\HH(I)\, r.
		\end{equation*}
		It follows that
		\begin{multline*}
		\mu(X(x,J\setminus G,r))\le \sum_{I\in\cB} \mu(X(x,I,r))\\
		 \le \sum_{I\in\cB}\mu(\cR_I) \lesssim M r\sum_{I\in\cB}\HH(I) = M \HH(J\setminus G)\, r.
		\end{multline*}
	\end{proof}
	
	\begin{lemma}\label{lem:filling gaps}
		If $\varepsilon=\varepsilon(M,C_0)$ is small enough, then for any $x\in E$ and $0<r<\infty$ we have
		\begin{equation}\label{eq:filling gaps}
		\mu(X(x,0.9J,r))\lesssim_{C_0} \mu(X(x,G,2r)).
		\end{equation}
		In particular, $\E_J^{int}(Q)\lesssim_{C_0}\E_G(Q)$, and so \eqref{eq:interior energy est} holds.
	\end{lemma}
		
	\begin{proof}
		If $X(x,0.9J,r)\cap E=\{x\}$, then there is nothing to prove, so suppose that $X(x,0.9J,r)\cap E\neq\{x\}$. 
		
		Let $y\in X(x,0.9J,r)\cap E\setminus \{x\}$, and let $0<r_0\le r/2$ be such that $y\in E\cap X(x,0.9J,r_0,2r_0)$.		
		Set $r_y = c\HH(J)\, r_0$ for some small absolute constant $c>0$, and observe that if $c$ is chosen small enough, then $B(y,r_y)\subset X(x,J, r_0/2,4r_0)$. 
		
		We use \lemref{lem:littlemeas} to estimate
		\begin{multline*}
		\mu(B(y,r_y)\cap X(x,J\setminus G,r_0/2,4r_0))\le \mu(X(x,J\setminus G,r_0/2,4r_0))\\
		\overset{\eqref{eq:littlemeas}}{\lesssim}M\varepsilon\HH(J) r_0\sim M\varepsilon r_y.
		\end{multline*}
		
		On the other hand, since $y\in E, r_y< r_0<\diam(E)=1$, and $B(y,r_y)\subset X(x,J, r_0/2,4r_0)$, we get from AD-regularity of $E$ that
		\begin{equation*}
		\mu(B(y,r_y)\cap X(x,J, r_0/2,4r_0))=\mu(B(y,r_y)) \gtrsim C_0^{-1} r_y.
		\end{equation*}
		The two estimates together give
		\begin{multline*}
		C_0^{-1} r_y\lesssim \mu(B(y,r_y)\cap X(x,J, r_0/2,4r_0))\\
		 = \mu(B(y,r_y)\cap X(x,G, r_0/2,4r_0)) + \mu(B(y,r_y)\cap X(x,J\setminus G, r_0/2,4r_0))\\
		 \le \mu(B(y,r_y)\cap X(x,G, r_0/2,4r_0)) + CM\varepsilon r_y.
		\end{multline*}
		Hence, assuming $\varepsilon=\varepsilon(M,C_0)$ small enough, we may absorb the second term on the right hand side to the left hand side, which gives
		\begin{multline}\label{eq:est2}
		\mu(B(y,r_y)\cap X(x,G,2r))\ge \mu(B(y,r_y)\cap X(x,G,r_0/2,4r_0)) \\
		\gtrsim C_0^{-1} r_y \sim_{C_0}\mu(B(y,r_y)).
		\end{multline}
		
		Now consider the family of balls
		\begin{equation*}
		\cB = \{B(y,r_y)\ :\ y\in X(x,0.9J,r)\cap E\setminus \{x\}\}.
		\end{equation*}
		By the $5r$-covering lemma, we may find a countable sub-collection $\cB' = \{B(y_i,r_{y_i})\}_{i\in\mathcal{I}}$ of pairwise disjoint balls such that $\{B(y_i,5r_{y_i})\}_{i\in\mathcal{I}}$ covers all of $X(x,0.9J,r)\cap E\setminus \{x\}$. Then,
		\begin{multline*}
		\mu(X(x,0.9J,r)\cap E)\le \mu\bigg(\bigcup_{i\in\mathcal{I}}B(y_i,5r_{y_i})\bigg)
		\le \sum_{i\in\mathcal{I}}\mu(B(y_i,5r_{y_i}))\\
		\sim_{C_0}\sum_{i\in\mathcal{I}}\mu(B(y_i,r_{y_i})) \overset{\eqref{eq:est2}}{\lesssim_{C_0}} \sum_{i\in\mathcal{I}}\mu(B(y_i,r_{y_i})\cap X(x,G,2r))\\
		 = \mu\bigg(\bigcup_{i\in\mathcal{I}}B(y_i,r_{y_i})\cap X(x,G,2r)\bigg)\le \mu(X(x,G,2r).
		\end{multline*}
	\end{proof}

	\subsection{Obtaining good cones}
	We will say that a (possibly truncated) cone $X$ is \emph{good} if it satisfies
	\begin{equation*}
	X\cap E = \varnothing.
	\end{equation*}
	Similarly, we will say that a rectangle $\cR$ is good if $\cR\cap E=\varnothing$.
	
	Having plenty of good cones and rectangles will be crucial for estimating the exterior energy $\EE^{ext}(Q)$ in Section \ref{sec:exterior energy}. In the lemma below we use \lemref{lem:filling gaps} and the $\BCE$-stopping condition to find many good cones.
	
	\begin{lemma}\label{lem:empty cones}
		If the $\BCE$-parameter $\delta=\delta(A,M,C_0)\in(0,1)$ is chosen small enough, then for all $R\in\Top,\, Q\in\Tree(R)\setminus\BCE(R),$ and $x\in A\cR_Q\cap E$ we have
		\begin{equation*}
		X(x,0.5J,A^{-1}\ELL(Q), A^2\ELL(R))\cap E=\varnothing.
		\end{equation*}
	\end{lemma}
	\begin{proof}
		Assume the contrary: let $Q\in\Tree(R)\setminus\BCE(R),\ x\in A\cR_Q\cap E$, and $y\in X(x,0.5J,A^{-1}\ELL(Q), A^2\ELL(R))\cap E$. 
		
		Let $P\in\Tree(R)\setminus\BCE(R)$ be such that $Q\subset P$ and $y\in X(x,0.5J,A^{-1}\ELL(P), A^2\ELL(P))$, so that in particular
		\begin{equation*}
		A^{-1}\ELL(P)\le|x-y|\le A^2\ELL(P).
		\end{equation*}
		
		
		Set	
		\begin{equation}\label{eq:r0}
		r_0 \coloneqq A^{-2} \ell(P) =  A^{-2}\HH(J)\ELL(P)\le A^{-1}\HH(J)|x-y|.
		\end{equation}
		We claim that if $A$ is chosen big enough, then for all $x'\in B(x,r_0)$ we have 
		\begin{equation}\label{eq:ballincone}
		B(y,r_0)\subset X(x',0.9J, 2A^2\ELL(P)).
		\end{equation}
		This is a simple geometric observation, see Figure \ref{fig:1}. The rigorous computation goes as follows: first, observe that if $x'\in B(x,r_0),\ y'\in B(y,r_0)$, then
		\begin{equation*}
		|x'-y'|\ge |x-y|-2 r_0\overset{\eqref{eq:r0}}{\ge} A\HH(J)^{-1}r_0 - 2r_0\ge \frac{{A}}{2\HH(J)}r_0.
		\end{equation*}
		Thus, using the fact that $y\in X(x,0.5J)$,
		\begin{multline*}
		|\pi_0(x')-\pi_0(y')|\le |\pi_0(x)-\pi_0(y)|+2r_0\overset{\eqref{eq:cone algebraic}}{\le}\sin\bigg(\frac{\HH(J)}{2}\pi\bigg)|x-y|+2r_0\\
		\le \sin\bigg(\frac{\HH(J)}{2}\pi\bigg)|x'-y'|+4r_0 \le \bigg(\sin\bigg(\frac{\HH(J)}{2}\pi\bigg) + \frac{8\HH(J)}{A}\bigg)|x'-y'|\\
		\le \sin\big(0.9\HH(J)\pi\big)|x'-y'|,
		\end{multline*}
		assuming $A$ large enough. This shows $y'\in X(x',0.9J)$. We also have $y'\in B(x', 2A^2\ELL(P))$ because 
		\begin{equation*}
		|x'-y'|\le |x-y|+2r_0\le A^2\ELL(P)+ 2A^{-2}\ell(P)\le 2 A^2\ELL(P).
		\end{equation*}
		This gives the claim \eqref{eq:ballincone}.
		
		\begin{figure}
			\def\svgwidth{6cm}
\begingroup%
  \makeatletter%
  \providecommand\color[2][]{%
    \errmessage{(Inkscape) Color is used for the text in Inkscape, but the package 'color.sty' is not loaded}%
    \renewcommand\color[2][]{}%
  }%
  \providecommand\transparent[1]{%
    \errmessage{(Inkscape) Transparency is used (non-zero) for the text in Inkscape, but the package 'transparent.sty' is not loaded}%
    \renewcommand\transparent[1]{}%
  }%
  \providecommand\rotatebox[2]{#2}%
  \newcommand*\fsize{\dimexpr\f@size pt\relax}%
  \newcommand*\lineheight[1]{\fontsize{\fsize}{#1\fsize}\selectfont}%
  \ifx\svgwidth\undefined%
    \setlength{\unitlength}{242.61502585bp}%
    \ifx\svgscale\undefined%
      \relax%
    \else%
      \setlength{\unitlength}{\unitlength * \real{\svgscale}}%
    \fi%
  \else%
    \setlength{\unitlength}{\svgwidth}%
  \fi%
  \global\let\svgwidth\undefined%
  \global\let\svgscale\undefined%
  \makeatother%
  \begin{picture}(1,1.54662345)%
    \lineheight{1}%
    \setlength\tabcolsep{0pt}%
    \put(0,0){\includegraphics[width=\unitlength,page=1]{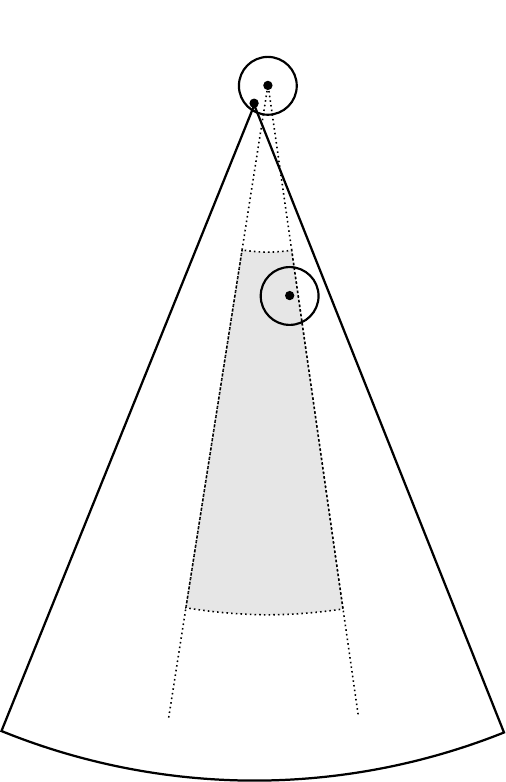}}%
    \put(0.6021402,1.38573019){\color[rgb]{0,0,0}\makebox(0,0)[lt]{\lineheight{1.25}\smash{\begin{tabular}[t]{l}$B(x,r_0)$\end{tabular}}}}%
    \put(0.708312,1.03130362){\color[rgb]{0,0,0}\makebox(0,0)[lt]{\lineheight{1.25}\smash{\begin{tabular}[t]{l}$B(y,r_0)$\end{tabular}}}}%
    \put(0,0){\includegraphics[width=\unitlength,page=2]{balls-cones.pdf}}%
    \put(0.28674719,1.38885294){\color[rgb]{0,0,0}\makebox(0,0)[lt]{\lineheight{1.25}\smash{\begin{tabular}[t]{l}$x'$\end{tabular}}}}%
    \put(0,0){\includegraphics[width=\unitlength,page=3]{balls-cones.pdf}}%
    \put(0.335949,1.50904463){\color[rgb]{0,0,0}\makebox(0,0)[lt]{\lineheight{1.25}\smash{\begin{tabular}[t]{l}$x$\end{tabular}}}}%
    \put(0,0){\includegraphics[width=\unitlength,page=4]{balls-cones.pdf}}%
  \end{picture}%
\endgroup%

			\caption{We have $B(y,r_0)\subset X(x',0.9J, 2A^2\ELL(P))$.}\label{fig:1}
		\end{figure}
		

		Since $x\in A\cR_P$ and $B(x,r_0)\subset 2A\cR_P$, we get from \lemref{lem:filling gaps} that
		\begin{align*}
		\E_G(P)\mu(P) &= \int_{2A\cR_P}\int_{A^{-1}\ELL(P)}^{A^3\ELL(P)} \frac{\mu(X(x',G,r))}{r}\, \frac{dr}{r}d\mu(x')\\
	&\overset{\eqref{eq:filling gaps}}{\gtrsim}_{C_0}\int_{2A\cR_P}\int_{2A^2\ELL(P)}^{4A^2\ELL(P)} \frac{\mu(X(x',0.9J,r))}{r}\, \frac{dr}{r}d\mu(x')\\
	&\ge \int_{B(x,r_0)}\int_{2A^2\ELL(P)}^{4A^2\ELL(P)} \frac{\mu(X(x',0.9J,r))}{r}\, \frac{dr}{r}d\mu(x')\\	
	&\gtrsim_A \int_{B(x,r_0)}\int_{2A^2\ELL(P)}^{4A^2\ELL(P)} \frac{\mu(B(y,r_0))}{\ELL(P)}\, \frac{dr}{r}d\mu(x')\\
	&\ge \frac{\mu(B(x,r_0))\mu(B(y,r_0))}{\ELL(P)} \gtrsim \frac{C_0^{-2}r_0^2}{\ELL(P)}
		\sim_{C_0,A} \frac{\ell(P)^2}{\ELL(P)} = \HH(J)\ell(P). 
		\end{align*}
		Hence,
		\begin{equation*}
		\E_G(P) \gtrsim_{C_0,A} \HH(J)\frac{\ell(P)}{\mu(P)}\gtrsim_M \HH(J).
		\end{equation*}
		Recall that $\E_G(P)\le\delta\HH(J)$ because $P\notin\BCE(R)$ (see the $\BCE$ stopping condition \eqref{eq:BCE def}). Assuming $\delta=\delta(A,M,C_0)$ small enough, we arrive at a contradiction.
	\end{proof}

	For brevity of notation, for $R\in\Top$ we define $\cT(R)=\Tree(R)\setminus\BCE(R)$ and
	\begin{equation*}
	\cT_k(R) = \cT(R)\cap\DD_k.
	\end{equation*}
	In the next two lemmas we show that for any integer $k\in\Z$, the family of intervals 
	\begin{equation*}
	\{\pi_0(\cR_P)\ :\ P\in \cT_k(R)\}
	\end{equation*}
	has bounded overlaps. In other words, if we fix a generation $\DD_k$, then the rectangles associated to cubes in $\cT_k(R)$ resemble a graph over the horizontal line $\ell_0$. This will be useful in Section \ref{sec:exterior energy}. Recall that $\DD_*$ was defined in Subsection \ref{subsec:corona}.

	\begin{lemma}\label{lem:projections disjoint}
	There exists an absolute constant $C>1$ such that the following holds. Suppose that $R\in\DD_*$ and $Q, P\in\DD(R)$ are such that $Q\neq P$, $\ell(Q)=\ell(P),$ and
	\begin{equation}\label{eq:empty cone1}
		X(z,0.5J,\rho\ELL(Q), \ELL(R))\cap E=\varnothing\quad\text{for all $z\in E\cap 2\cR_Q$.}
	\end{equation} 
	If $\pi_0(\cR_Q)\cap\pi_0(\cR_P)\neq\varnothing$, then $\cR_P\subset C\cR_Q$.
\end{lemma}
Note that since $\rho=1/1000$ is much larger than $A^{-1}=A(C_0,M)^{-1}$, the assumptions above are in particular satisfied for any $Q,P\in\DD_k\cap\Tree(R)\setminus\BCE(R)$ by \lemref{lem:empty cones}.
\begin{proof}
	Let $y_Q \in Q, y_P\in P,$ and suppose there exists $z_Q\in \cR_Q$ and $z_P\in \cR_P$ such that $\pi_0(z_Q)=\pi_0(z_P)$. 
	Then, we have
	\begin{multline*}
		|\pi_0(y_Q)-\pi_0(y_P)| = |\pi_0(y_Q-z_Q)-\pi_0(y_P-z_P) - \pi_0(z_P-z_Q)| \\
		\le |\pi_0(y_Q-z_Q)| + |\pi_0(y_P-z_P)| + |\pi_0(z_P-z_Q)|\le \ell(Q) + \ell(P) + 0 = 2\ell(Q).
	\end{multline*}
	We claim that $|\pi_0^\perp(y_Q) - \pi_0^\perp(y_P)|\le C'\ELL(Q)$ for some big absolute $C'>1$. Indeed, if that was not the case, then the previous computation gives
	\begin{equation*}
	|\pi_0(y_Q)-\pi_0(y_P)| \le 2\ell(Q) = 2\HH(J)\ELL(Q) \le \frac{2\HH(J)}{C'}|y_Q-y_P|.
	\end{equation*}
	Taking $C'>1$ large enough, we arrive at
	\begin{equation*}
		y_P\in X(y_Q,0.5J,\rho\ELL(Q), \ELL(R)),
	\end{equation*}
	which is a contradiction with \eqref{eq:empty cone1}. Hence, $|\pi_0^\perp(y_Q) - \pi_0^\perp(y_P)|\le C'\ELL(Q)$.
	
	Recall that $x_Q$ is the center of $\cR_Q$. It follows easily from the estimates above that for any $x\in \cR_P$
	\begin{multline*}
	|\pi_0(x)-\pi_0(x_Q)|\le |\pi_0(x)-\pi_0(y_P)| + |\pi_0(y_P)-\pi_0(y_Q)| + |\pi_0(y_Q)-\pi_0(x_Q)|\\
	\le \ell(P) + 2\ell(Q) + \ell(Q) = 4\ell(Q),
	\end{multline*}
	and
	\begin{multline*}
	|\pi_0^\perp(x) - \pi_0^\perp(x_Q)|\le |\pi_0^\perp(x) - \pi_0^\perp(y_P)|+|\pi_0^\perp(y_P) - \pi_0^\perp(y_Q)|+|\pi_0^\perp(y_Q) - \pi_0^\perp(x_Q)|\\
	\le \ELL(P) + C'\ELL(Q)+\ELL(Q)\lesssim \ELL(Q).	
	\end{multline*}
	Thus, $\cR_P\subset C\cR_Q$ for some absolute $C>1$.
\end{proof}

	Recall that that for $Q\in\DD_k$ we have $\ell(Q)=4\rho^k$.
	\begin{lemma}
		Let $R\in\Top$ and $k\ge 0$. Then, the family of intervals $\{\pi_0(\cR_P)\}_{P\in \cT_k(R)}$ has bounded overlaps, i.e.
		\begin{equation}\label{eq:bounded-intersection}
			\sum_{P\in \cT_k(R)} \one_{\pi_0(\cR_P)}(x)\lesssim 1\quad\text{for all $x\in\R$}.
		\end{equation}
		In particular, for any interval $K\subset \R$ we have
		\begin{equation}\label{eq:bounded number of intervals}
			\#\big\{P\in \cT_k(R)\ :\ \pi_0(\cR_P)\subset K\big\}\lesssim \frac{\HH(K)}{\rho^k}.
		\end{equation}
	\end{lemma}
	\begin{proof}
		Fix $Q\in\cT_k(R)$. Suppose that $P\in\cT_k(R)$ satisfies $\pi_0(\cR_Q)\cap\pi_0(\cR_P)\neq\varnothing$. We know from \lemref{lem:empty cones} that $Q$ and $P$ satisfy \eqref{eq:empty cone1}, and so it follows \lemref{lem:projections disjoint} that $\cR_P\subset C\cR_Q$. It remains to observe that
		\begin{equation*}
		\#\{ P\in\cT(R)\cap\DD_k \ :\ \cR_P\subset C\cR_Q \}\lesssim_C 1.
		\end{equation*}
		This gives \eqref{eq:bounded-intersection}.
		
		To see \eqref{eq:bounded number of intervals}, we compute
		\begin{multline*}
		\#\big\{P\in \cT_k(R)\ :\ \pi_0(\cR_P)\subset K\big\} \le \sum_{P\in \cT_k(R)} \frac{1}{\rho^k}\int_K \one_{\pi_0(\cR_P)}(x)\, dx\\
		= \frac{1}{\rho^k}\int_K \sum_{P\in \cT_k(R)}\one_{\pi_0(\cR_P)}(x)\, dx \overset{\eqref{eq:bounded-intersection}}{\lesssim}\frac{\HH(K)}{\rho^k}.
		\end{multline*}
	\end{proof}

	\section{Estimating exterior energy}\label{sec:exterior energy}
	Recall that
	\begin{equation*}
	\widetilde{\E}_J^{ext}(Q) =\frac{1}{\mu(Q)} \int_{Q} \frac{\mu(X(x,3J\setminus 0.5J,\rho\ELL(Q), \ELL(Q)))}{\ELL(Q)}\, d\mu(x).
	\end{equation*}
	Our goal is to prove the following.
	\begin{lemma}\label{lem:exterior energy estimate}
		If $A=A(C_0,M)$ is chosen large enough, then for any $R\in\Top$ we have
		\begin{equation*}
		\sum_{Q\in\Tree(R)} \EE^{ext}(Q)\mu(Q) \lesssim_{C_0,M} \HH(J)\mu(R).
		\end{equation*}
	\end{lemma}

	This estimate will follow from the key geometric lemma below. In order to state it, we introduce some notation.
	
	\begin{definition}
		For $R\in\DD_*$ we define $U(R)\subset\R$ as
		\begin{align*}
		U(R)&\coloneqq \pi_0(A\cR_R)\setminus\pi_0(A\cR_R\cap E)\\
		&= \big[\pi_{0}(x_R)-A\ell(R)/2,\ \pi_{0}(x_R)+A\ell(R)/2\big] \setminus\pi_0(A\cR_R\cap E).
		\end{align*}
		Denote by $\cK(R)$ the family of connected components of $U(R)$. Since $E$ is closed, the elements of $\cK(R)$ are intervals. We will call them \emph{gaps in $\pi_{0}(A\cR_R\cap E)$}.
	\end{definition}
	Since the gaps are disjoint, and they have positive length, we get that $\cK(R)$ is at most countable, and also
	\begin{equation}\label{eq:gap lengths}
	\sum_{K\in\cK(R)}\HH(K)\le \HH(U(R))\le \HH(\pi_0(A\cR_R)) = A\ell(R).
	\end{equation}
	Given $0<r<\ell(R)$ we define the collection of gaps with length comparable to $r$ as
	\begin{equation*}
		\cK(R,r) = \{K\in\cK(R)\ :\ A^{-1}r\le\HH(K)\le Ar \}.
	\end{equation*}
	
	\begin{definition}
		For $R\in\DD_*$, we define the family $\Bad(R)\subset\DD(R)$ as the family of cubes $Q\in\DD(R)$ for which there exists $x\in Q$ such that
		\begin{equation*}
		X(x,3J\setminus 0.5J, \rho\ELL(Q),\ELL(Q))\cap E\neq\varnothing.
		\end{equation*}
	\end{definition}
	Observe that if $Q\notin\Bad(R)$, then
	\begin{equation*}
	\widetilde{\E}_J^{ext}(Q) =\frac{1}{\mu(Q)} \int_{Q} \frac{\mu(X(x,3J\setminus 0.5J,\rho\ELL(Q), \ELL(Q)))}{\ELL(Q)}\, d\mu(x) = 0.
	\end{equation*}

	The following is the key geometric lemma of this article.
	\begin{lemma}\label{lem:key geometric lemma}
		If $A=A(C_0,M)$ is chosen large enough, then the following holds. Suppose that $R\in\DD_*$ and $Q\in\DD(R)$ are such that
		\begin{equation}\label{eq:empty cones2}
			X(z,0.5J,A^{-1}\ELL(Q), A^2\ELL(R))\cap E=\varnothing\quad\text{for all $z\in A\cR_Q\cap E$.}
		\end{equation}
		If $Q\in\Bad(R)$, then there is a gap $K\in\cK(R,\ell(Q))$ such that 
		\begin{equation*}
			\pi_0(\cR_Q)\subset A^3K.
		\end{equation*}
	\end{lemma}

	We defer the proof to the next section. Let us show how \lemref{lem:exterior energy estimate} follows from \lemref{lem:key geometric lemma}.
	\begin{proof}[Proof of \lemref{lem:exterior energy estimate}]
		Let $R\in\Top$. Our goal is to prove
		\begin{equation*}
		\sum_{Q\in\Tree(R)} \EE^{ext}(Q)\mu(Q) \lesssim_{C_0,M} \HH(J)\mu(R).
		\end{equation*}
		Recall that $\cT(R) = \Tree(R)\setminus\BCE(R),\ \cT_k(R)=\cT(R)\cap\DD_k$ If $Q\notin\Bad(R)$, then $\EE^{ext}(Q)=0$ trivially, and so it suffices to show
		\begin{equation}\label{eq:goal2}
		\sum_{Q\in\cT(R)\cap\Bad(R)} \EE^{ext}(Q)\mu(Q) +  \sum_{Q\in\BCE(R)} \EE^{ext}(Q)\mu(Q) \lesssim_{C_0,M} \HH(J)\mu(R).
		\end{equation} 
		
		Observe that for any $x\in E$ we have
		\begin{equation*}
		\mu(X(x,3J\setminus 0.5J,\rho\ELL(Q), \ELL(Q)))\le \mu(\cR(x,3\ell(Q)))\overset{\eqref{eq:ADRrectangles}}{\lesssim} M\ell(Q),
		\end{equation*}
		and so for any $Q\in\DD_*$
		\begin{multline*}
		\widetilde{\E}_J^{ext}(Q)\mu(Q)=\int_{Q} \frac{\mu(X(x,3J\setminus 0.5J,\rho\ELL(Q), \ELL(Q)))}{\ELL(Q)}\, d\mu(x)\\
		\lesssim \frac{M\ell(Q)}{\ELL(Q)}\,\mu(Q) = M\HH(J)\mu(Q).
		\end{multline*}
		It follows that
		\begin{multline*}
		\sum_{Q\in\cT(R)\cap\Bad(R)} \EE^{ext}(Q)\mu(Q) +  \sum_{Q\in\BCE(R)} \EE^{ext}(Q)\mu(Q)\\
		\lesssim M\HH(J)\bigg(\sum_{Q\in\cT(R)\cap\Bad(R)} \mu(Q) +  \sum_{Q\in\BCE(R)} \mu(Q)\bigg).
		\end{multline*}
		Thus, to reach \eqref{eq:goal2}, it suffices to show that the two sums on the right hand side above are bounded by $C(C_0,M)\mu(R)$. This is immediate for the second sum:
		\begin{equation*}
		\sum_{Q\in\BCE(R)}\mu(Q)\le\mu(R).
		\end{equation*}
		
		What remains to show is that
		\begin{equation}\label{eq:goal1}
		\sum_{Q\in\cT(R)\cap\Bad(R)} \mu(Q)\lesssim_{C_0,M} \mu(R).
		\end{equation}
		Let $Q\in\cT(R)\cap\Bad(R)\subset\Tree(R)\setminus\BCE(R)$.
		 By \lemref{lem:empty cones}, $R$ and $Q$ satisfy the empty cone assumption \eqref{eq:empty cones2}, and so we may use \lemref{lem:key geometric lemma} to conclude that there is a gap $K\in\cK(R,\ell(Q))$ such that $\pi_0(\cR_Q)\subset A^3K.$ Hence,
		\begin{align*}
			\sum_{Q\in\cT(R)\cap\Bad(R)} \mu(Q)& =\sum_{k\ge 0} \sum_{Q\in\cT_k(R)\cap\Bad(R)} \mu(Q)\\
			&\le \sum_{k\ge 0}\sum_{K\in\cK(R,4\rho^k)}\sum_{\substack{Q\in\cT_k(R),\\ \pi_0(\cR_Q)\subset A^3K}} \mu(Q)\\
			&\lesssim \sum_{k\ge 0}\sum_{K\in\cK(R,4\rho^k)}\sum_{\substack{Q\in\cT_k(R),\\ \pi_0(\cR_Q)\subset A^3K}} M\ell(Q)\\
			&\overset{\eqref{eq:bounded number of intervals}}{\lesssim} \sum_{k\ge 0}\sum_{K\in\cK(R,4\rho^k)} M\rho^k \frac{\HH(A^3K)}{\rho^k}
			\\
			&\sim_{A,M} \sum_{k\ge 0}\sum_{K\in\cK(R,4\rho^k)} \HH(K)
			\sim_{A} \sum_{K\in\cK(R)} \HH(K) \overset{\eqref{eq:gap lengths}}{\lesssim}_A\ell(R).
		\end{align*}
		Since $A=A(C_0,M)$ and $\mu(R)\gtrsim C_0^{-1}\ell(R)$, this gives the desired estimate \eqref{eq:goal1}.
%
	\end{proof}

	\section{Proof of the key geometric lemma}\label{sec:keygeometric} In this section we prove \lemref{lem:key geometric lemma}. 

		\subsection{Preliminaries}
		Suppose that $R\in\DD_*$ and $Q\in\DD(R)$ are as in the assumptions of \lemref{lem:key geometric lemma}, so that they satisfy
		\begin{equation}\label{eq:empty cones}
		X(z,0.5J,A^{-1}\ELL(Q), A^2\ELL(R))\cap E=\varnothing\quad\text{for all $z\in A\cR_Q\cap E$,}
		\end{equation}
		and assume that $Q\in\Bad(R)$, which means that there exists $x\in Q$ such that 
		\begin{equation*}
		X(x,3J\setminus 0.5J,\rho\ELL(Q), \ELL(Q))\cap E\neq\varnothing.
		\end{equation*}
		Let $y\in X(x,3J\setminus 0.5J,\rho\ELL(Q), \ELL(Q))\cap E$. 		
		See Figure \ref{fig:KGL} for an overview of our setup.
		
		The plan is as follows. We want to find a gap $K\in\cK(R,\ell(Q))$ such that 
		\begin{equation*}
		\pi_0(\cR_Q)\subset A^3K.
		\end{equation*}
		To achieve this, we will find a rectangle $\mathcal{Y}$ satisfying $\mathcal{Y}\cap E=\varnothing$ (in our terminology: ``$\cA$ is a good rectangle'') of size roughly $\ell(Q)\times\ELL(R)$, such that $\pi_0^\perp(\mathcal{Y})\supset\pi_0^\perp(A\cR_R)$, and such that $\mathcal{Y}$ lies between $x$ and $y$, in the sense that $\pi_0(x)$ and $\pi_0(y)$ lie on different sides of the interval $\pi_0(\mathcal{Y})$. See the yellow rectangle in Figure \ref{fig:KGL}. The properties above tell us that
		\begin{equation*}
		\pi_0(\mathcal{Y})\cap\pi_0(A\cR_R\cap E)=\varnothing,
		\end{equation*}
		so that $\pi_0(\mathcal{Y})$ is contained in some interval $K\in\cK(R)$. One can also see that $K$ necessarily satisfies $\HH(K)\sim_A \ell(Q)$, so that $K\in\cK(R,\ell(Q))$. This will be our desired gap.

		\begin{remark}
			It is instructive to consider the following hypothetical counterexample to what we are aiming to prove. Suppose that $R=E$ is a segment of length $\sim\ELL(R)$ containing $x$ and $y$. It is easy to see that for every $z\in E$ we have $X(z, 0.5J)\cap E = \varnothing$, which is even better than \eqref{eq:empty cones}. At the same time, the projection $\pi_0(A\mathcal{R}_R\cap E) = \pi_0(R)$ is an interval of length $\sim\ell(R)$, and one cannot hope to find a gap $K$ lying between $\pi_0(x)$ and $\pi_0(y)$. 
			
			This does not contradict Lemma \ref{lem:key geometric lemma} for the following reason. Observe that in this example the projected measure $\pi_{\theta_0}\mathcal{H}^1|_{E}$ is a uniform measure on a segment of length $\sim\ell(R)$ with total mass $\sim\ELL(R)=\HH(J)^{-1}\ell(R)$. Using the upper bound on the length of $\HH(J)$ from assumption (b) in Proposition \ref{prop:main prop}, this gives
			\begin{equation*}
				\|\pi_{\theta_0}\mathcal{H}^1|_{E}\|_{L^\infty}\sim \HH(J)^{-1}\ge c_1^{-1}M
			\end{equation*}
			for some small absolute $c_1$ that we choose in Lemma \ref{lem:ddd} below. Since $c_1^{-1}$ is very large, we get that the set $E$ does not satisfy our underlying assumption $\|\pi_{\theta_0}\mathcal{H}^1|_{E}\|_{L^\infty}\le M$. Thus, Lemma \ref{lem:key geometric lemma} cannot be applied to this set.
		\end{remark}

		\begin{figure}
			\def\svgwidth{4cm}
\begingroup%
  \makeatletter%
  \providecommand\color[2][]{%
    \errmessage{(Inkscape) Color is used for the text in Inkscape, but the package 'color.sty' is not loaded}%
    \renewcommand\color[2][]{}%
  }%
  \providecommand\transparent[1]{%
    \errmessage{(Inkscape) Transparency is used (non-zero) for the text in Inkscape, but the package 'transparent.sty' is not loaded}%
    \renewcommand\transparent[1]{}%
  }%
  \providecommand\rotatebox[2]{#2}%
  \newcommand*\fsize{\dimexpr\f@size pt\relax}%
  \newcommand*\lineheight[1]{\fontsize{\fsize}{#1\fsize}\selectfont}%
  \ifx\svgwidth\undefined%
    \setlength{\unitlength}{174.11039365bp}%
    \ifx\svgscale\undefined%
      \relax%
    \else%
      \setlength{\unitlength}{\unitlength * \real{\svgscale}}%
    \fi%
  \else%
    \setlength{\unitlength}{\svgwidth}%
  \fi%
  \global\let\svgwidth\undefined%
  \global\let\svgscale\undefined%
  \makeatother%
  \begin{picture}(1,3.65503752)%
    \lineheight{1}%
    \setlength\tabcolsep{0pt}%
    \put(0.67952043,0.20573204){\color[rgb]{0,0,0}\makebox(0,0)[lt]{\lineheight{1.25}\smash{\begin{tabular}[t]{l}$A\cR_R$\end{tabular}}}}%
    \put(0,0){\includegraphics[width=\unitlength,page=1]{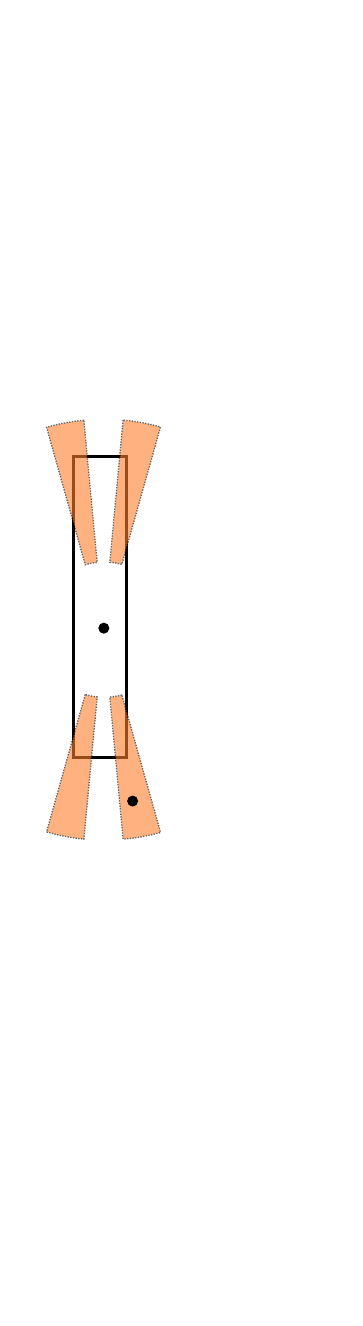}}%
    \put(0.03039063,1.8603204){\color[rgb]{0,0,0}\makebox(0,0)[lt]{\lineheight{1.25}\smash{\begin{tabular}[t]{l}$x$\end{tabular}}}}%
    \put(0.49047236,1.56501366){\color[rgb]{0,0,0}\makebox(0,0)[lt]{\lineheight{1.25}\smash{\begin{tabular}[t]{l}$y$\end{tabular}}}}%
    \put(0,0){\includegraphics[width=\unitlength,page=2]{KGL.pdf}}%
  \end{picture}%
\endgroup%

			\caption{The big white rectangle is $A\cR_R$, the small white rectangle is $\cR_Q$, the orange double-truncated cone is $X(x,3J\setminus 0.5J,\rho\ELL(Q), \ELL(Q))$, the yellow rectangle is the desired good rectangle $\mathcal{Y}$.}\label{fig:KGL}
		\end{figure}
	
		\vspace{1em}
		
		The double truncated cone $X(x,3J\setminus 0.5J,\rho\ELL(Q), \ELL(Q))$ has 4 connected components (see the orange cone in Figure \ref{fig:KGL} or \figref{fig:rectangleP}). Without loss of generality, we may assume that $y$ lies in the lower right connected component, so that $\pi_0(x)<\pi_{0}(y)$ and $\pi_{0}^\perp(x)>\pi_{0}^\perp(y)$ (the proof for other cases is completely analogous). Note that, since $y\in X(x,3J\setminus 0.5J,\rho\ELL(Q), \ELL(Q))$, we have
		\begin{equation*}
		\pi_{0}(y)-\pi_{0}(x)\sim \ell(Q),
		\end{equation*}
		and
		\begin{equation*}
		\pi_{0}^\perp(x)-\pi_{0}^\perp(y)\sim \ELL(Q).
		\end{equation*}
		
			
		\subsection{Finding a leftist rectangle}
		Recall that our desired good rectangle $\mathcal{Y}$ will be of size roughly $\ell(Q)\times\ELL(R)$ and will satisfy $\pi_0^\perp(\mathcal{Y})\supset\pi_0^\perp(A\cR_R)$. Note that any good cone arising from \eqref{eq:empty cones} already \emph{almost} contains a rectangle with these properties, except for a missing $\ell(Q)\times\ELL(Q)$ rectangle close to the center of the cone (see the red cone in Figure \ref{fig:rectangleA}). Our goal is to find an auxiliary good rectangle $\mathcal{B}$ of size roughly $\ell(Q)\times\ELL(Q)$, which will fill the missing piece of the good cone. See the blue rectangle in Figure \ref{fig:rectangleA}. 
		
		The good rectangle $\cB$ will be contained in something we called ``a leftist rectangle''. In order to define it, we first consider the rectangle
		\begin{equation*}
		\cG \coloneqq \bigg\{z\in\R^2\, :\, \pi_0(x)\le\pi_0(z)\le\pi_{0}(y),\ |\pi_{0}^\perp(z) - \pi_{0}^\perp(y)|\le \frac{|\pi_0^\perp(x)-\pi_0^\perp(y)|}{2}\bigg\},
		\end{equation*}
		see the gray rectangle in Figure \ref{fig:rectangleP}. Note that $\ell(\cG)=|\pi_0(x)-\pi_0(y)|\sim\ell(Q),\, \ELL(\cG)=|\pi_0^\perp(x)-\pi_0^\perp(y)|\sim\ELL(Q)$, and the mid-point of its right edge is $y$.
		
		\begin{figure}
			\def\svgwidth{3cm}
\begingroup%
  \makeatletter%
  \providecommand\color[2][]{%
    \errmessage{(Inkscape) Color is used for the text in Inkscape, but the package 'color.sty' is not loaded}%
    \renewcommand\color[2][]{}%
  }%
  \providecommand\transparent[1]{%
    \errmessage{(Inkscape) Transparency is used (non-zero) for the text in Inkscape, but the package 'transparent.sty' is not loaded}%
    \renewcommand\transparent[1]{}%
  }%
  \providecommand\rotatebox[2]{#2}%
  \newcommand*\fsize{\dimexpr\f@size pt\relax}%
  \newcommand*\lineheight[1]{\fontsize{\fsize}{#1\fsize}\selectfont}%
  \ifx\svgwidth\undefined%
    \setlength{\unitlength}{148.27637879bp}%
    \ifx\svgscale\undefined%
      \relax%
    \else%
      \setlength{\unitlength}{\unitlength * \real{\svgscale}}%
    \fi%
  \else%
    \setlength{\unitlength}{\svgwidth}%
  \fi%
  \global\let\svgwidth\undefined%
  \global\let\svgscale\undefined%
  \makeatother%
  \begin{picture}(1,3.35789447)%
    \lineheight{1}%
    \setlength\tabcolsep{0pt}%
    \put(0,0){\includegraphics[width=\unitlength,page=1]{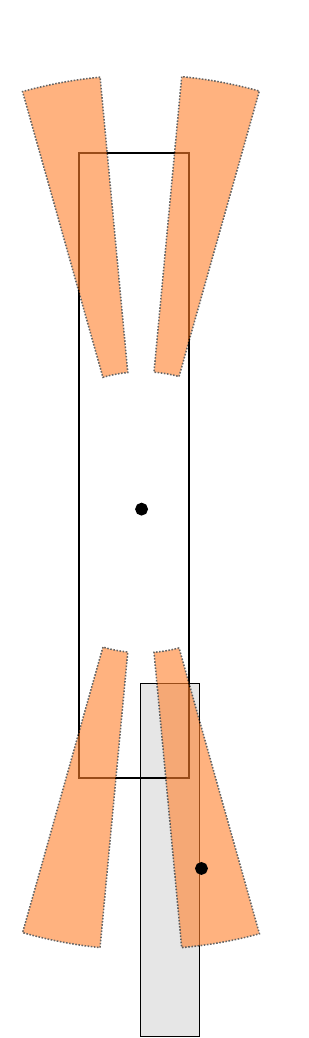}}%
    \put(0.32177261,1.66756741){\color[rgb]{0,0,0}\makebox(0,0)[lt]{\lineheight{1.25}\smash{\begin{tabular}[t]{l}$x$\end{tabular}}}}%
    \put(0.85584353,0.67842814){\color[rgb]{0,0,0}\makebox(0,0)[lt]{\lineheight{1.25}\smash{\begin{tabular}[t]{l}$y$\end{tabular}}}}%
    \put(0,0){\includegraphics[width=\unitlength,page=2]{rectangleP.pdf}}%
  \end{picture}%
\endgroup%

			\caption{The white rectangle is $\cR_Q$, the gray rectangle is $\cG$, and the orange double-truncated cone is $X(x,3J\setminus 0.5J,\rho\ELL(Q), \ELL(Q))$.}\label{fig:rectangleP}
		\end{figure}
			
		Let $N>1$ be a large integer satisfying
		\begin{equation}\label{eq:Ndef}
		N \sim MC_0,
		\end{equation}
		whose precise value will be fixed later on.
		
		We divide $\cG$ into $2N+1$ sub-rectangles $\cG_{-N},\dots,\cG_0,\dots,\cG_N$ such that $\ell(\cG_i)=\ell(\cG)=|\pi_0(x)-\pi_0(y)|$ and $\ELL(\cG_i)=\ELL(\cG)/(2N+1)=|\pi_0^\perp(x)-\pi_0^\perp(y)|/(2N+1)$. We enumerate them in such a way that each $\cG_i$ is on top of $\cG_{i-1}$, and $\cG_0$ is the rectangle containing $y$. See the left hand side of Figure \ref{fig:rectanglesPi}. In formulas,
		\begin{multline*}
		\cG_i\coloneqq \bigg\{z\in\R^2\, :\, \pi_0(x)\le\pi_0(z)\le\pi_{0}(y),\\
		\frac{(2i-1)\ELL(\cG)}{2(2N+1)}\le\pi_{0}^\perp(z) - \pi_{0}^\perp(y)\le \frac{(2i+1)\ELL(\cG)}{2(2N+1)}\bigg\}.
		\end{multline*}
		
		It is not immediately clear that $\ell(\cG_i)$ and $\ELL(\cG_i)$ as we defined them satisfy $\ell(\cG_i)\le\ELL(\cG_i)$, and that $\cG_i$'s look as portrayed in Figure \ref{fig:rectanglesPi}, as opposed to being very flat. We check this in the lemma below.
		\begin{lemma}\label{lem:ddd}
			We have $\ell(\cG_i)\le\ELL(\cG_i)$.
		\end{lemma}
	\begin{proof}
		Recall that $\ell(\cG_i)=\ell(\cG)\sim\ell(Q)$, and 
		\begin{equation}\label{eq:ELLPi}
		\ELL(\cG_i) = \frac{\ELL(\cG)}{2N+1} \sim \frac{\ELL(Q)}{N} = \frac{\HH(J)^{-1}\ell(Q)}{N}= \frac{\ell(\cG_i)}{\HH(J)N}\overset{\eqref{eq:Ndef}}{\sim} \frac{\ell(\cG_i)}{\HH(J)MC_0}.
		\end{equation}
		Assumption (b) of \propref{prop:main prop} stated that $\HH(J)\le c_1C_0^{-1} M^{-1}$, where $c_1>0$ is a small absolute constant. Assuming $c_1$ to be small enough, the above estimates give
		\begin{equation}\label{eq:Pitall}
		\ELL(\cG_i)\ge \ell(\cG_i).
		\end{equation}
	\end{proof}
		
		\begin{figure}
			\def\svgwidth{8cm}
			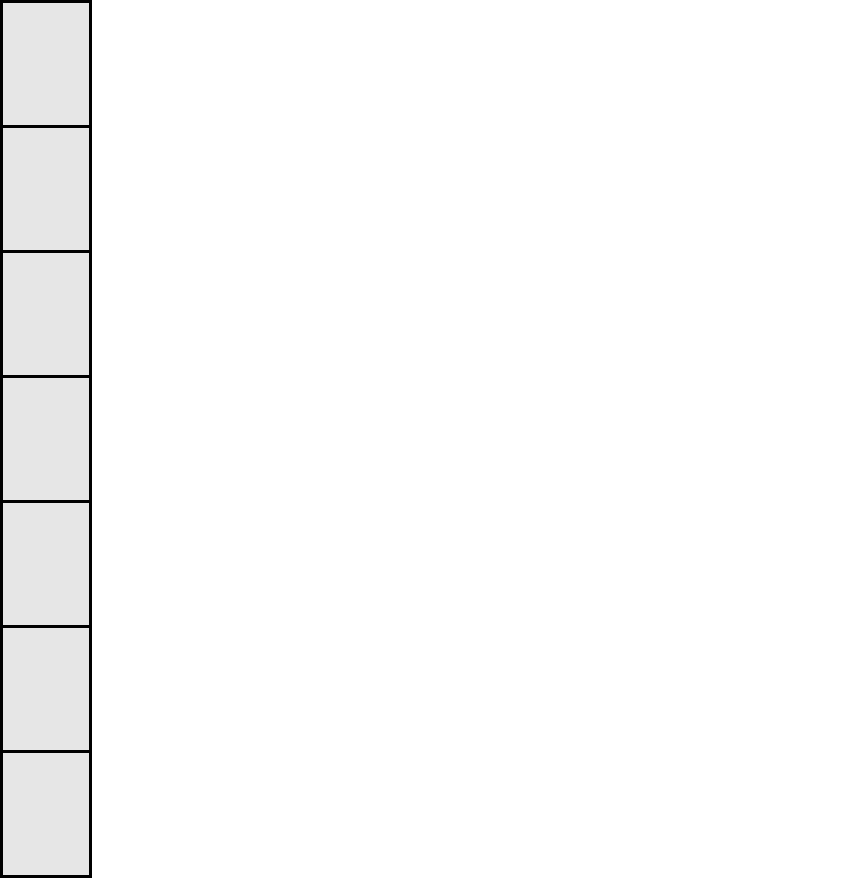
			\caption{On the left, the rectangle $\cG$ subdivided into subrectangles $\cG_i$ for $N=3$. On the right, 3 subrectangles $\cG_{i-1},\, \cG_i,\, \cG_{i+1}$. The black curves represent the set $E$. Since $\cG_i\cap E\neq\varnothing$ and $\cG_{i+1}\cap E\neq\varnothing$, the corresponding leftmost points $z_i$ and $z_{i+1}$ are well-defined. Note that $\cG_i$ is a leftist rectangle: $\cG_i\prec\cG_{i-1}$ because $\cG_{i-1}\cap E=\varnothing$, and $\cG_i\prec\cG_{i+1}$ because $\pi_0(z_i)\le\pi_0(z_{i+1})$.}\label{fig:rectanglesPi}
		\end{figure}		
		
		The following three definitions are easier to digest together with the right hand side of \figref{fig:rectanglesPi}. 
		\begin{definition}\label{def:leftmost}
			For each $\cG_i$ with $\cG_i\cap E\neq\varnothing$, let $z_i\in \cG_i\cap E$ be a point such that
			\begin{equation*}
			\pi_0(z_i)=\inf_{z\in \cG_i\cap E}\pi_0(z).
			\end{equation*}
			We will call $z_i$ the \emph{leftmost point of $\cG_i\cap E.$} Note that the left-most point is well-defined because $\cG_i$ and $E$ are closed. It might be non-unique, but we do not care.
		\end{definition}		
		
		\begin{definition}\label{def:relation}
			If $-N\le i,j\le N$ and $\cG_i\cap E\neq\varnothing$, then we will write {$\cG_i\prec\cG_j$} if either $\cG_j\cap E=\varnothing$ or $\pi_0(z_i)\le \pi_0(z_j)$. In other words, $\cG_i\prec\cG_j$ means that there is no point of $\cG_j\cap E$ to the left of $z_i$.
		\end{definition}
		
		\begin{definition}\label{def:goodrec}
			For $-N+1\le i\le N-1$, we will say that $\cG_i$ is a \emph{leftist rectangle} if $\cG_i\cap E\neq\varnothing$ and we have $\cG_i\prec\cG_{i-1}$ and $\cG_i\prec\cG_{i+1}$. That is, the point $z_i$ is the leftmost point of $(\cG_{i-1}\cup\cG_i\cup\cG_{i+1})\cap E$.
		\end{definition}

		\begin{lemma}\label{lem:good-rectangle}
			There exists $-N+1\le i\le N-1$ such that $\cG_i$ is a leftist rectangle.
		\end{lemma} 
		\begin{proof}
			Suppose the opposite, so that none of the rectangles is leftist. In particular, $\cG_0$ is not leftist. This means that either $\cG_0\cap E=\varnothing$, or for some $i\in\{-1,1\}$ we have $\cG_i\prec \cG_0$. Since $y\in\cG_0\cap E$, the second alternative holds. Without loss of generality assume that $\cG_1\prec \cG_0$. 
			
			Since $\cG_1$ is not leftist, but $\cG_1\prec\cG_0$, we get that $\cG_2\prec\cG_1$. In particular, $\cG_2\cap E\neq\varnothing$. Continuing in this way, we get for $1\le j\le N-1$ that $\cG_{j+1}\prec\cG_j\prec\cG_{j-1}$. In particular, for all $1\le j\le N$ we have $z_j\in\cG_j\cap E\neq\varnothing$. 
			
			Let $1\le j\le N$. By \eqref{eq:Pitall}, we have $B(z_j,\ell(\cG_j))\subset 3\cG_j$, and so
			\begin{equation*}
				\mu(3\cG_j)\ge \mu(B(z_j,\ell(\cG_j)))\ge C_0^{-1} \ell(\cG_j).
			\end{equation*}
			Since the rectangles $\{3\cG_j\}_{j=1}^N$ have bounded overlap, and they are all contained in $3\cG$, we get that
			\begin{equation}\label{eq:est5}
				\mu(3\cG)\gtrsim \sum_{j=1}^N \mu(3\cG_j)\ge \sum_{j=1}^N C_0^{-1} \ell(\cG_j) = NC_0^{-1} \ell(\cG).
			\end{equation}
		Recall that $\ell(\cG)=|x-y|$ and $\ELL(\cG)=|\pi_0(x)-\pi_0(y)|\sim\HH(J)^{-1}\ell(\cG)$, so that $3\cG\subset \cR(y,C\ell(\cG))$ for some absolute constant $C>1$. 
		
		Now is one of the key points where we use the $L^\infty$-estimate for projections. Recall that our assumption $\|\pi_{\theta_0}^\perp\mu\|_\infty \le M$ implied the upper bound on $\mu$-measure of rectangles \eqref{eq:ADRrectangles}. This gives
		\begin{equation}\label{eq:est6}
			\mu(3\cG)\le \mu(\cR(y,C\ell(\cG)))\lesssim M\ell(\cG).
		\end{equation}
		Let us compare this with the lower bound \eqref{eq:est5}. In the definition of $N$ \eqref{eq:Ndef} we assumed $N\sim MC_0$. Let $N= \lceil C'MC_0\rceil$, where $C'>1$ is a big absolute constant. Pitting \eqref{eq:est5} against \eqref{eq:est6} and choosing $C'>1$ large enough, we reach a contradiction.
		\end{proof}	
		
		The combination of \lemref{lem:good-rectangle} and the following lemma will complete the proof of the key geometric lemma. 
		\begin{lemma}\label{lem:why-good-rectangle}
			If $\cG_i$ is a leftist rectangle, then $\pi_0(z_i)$ is the right endpoint of some gap $K\in\cK(R,\ell(Q))$ with $\pi_0(\cR_Q)\subset A^3K$.
		\end{lemma}		
		We divide the proof of \lemref{lem:why-good-rectangle} into several steps.			
		\subsection{Small good rectangle $\cB$}
		Assume that $\cG_i$ is a leftist rectangle. We define
		\begin{align}\label{eq:defG}
		\cB&\coloneqq \{z\in \cG_{i-1}\cup\cG_i\cup\cG_{i+1}\, :\, \pi_0(z)\le\pi_{0}(z_i)\},\\
		&= \{z\in \cG_{i-1}\cup\cG_i\cup\cG_{i+1}\, :\, \pi_0(x)\le\pi_0(z)\le\pi_{0}(z_i)\},
		\end{align}
		see the blue rectangle in \figref{fig:rectangletildePi}. A priori it might happen that $\pi_0(z_i)=\pi_0(x)$, in which case $\cB$ would be a degenerate rectangle (a segment). We show in \lemref{lem:tildePiwide} below that this is not the case.
		
		Note that
		\begin{equation*}
		\ELL(\cB)=\ELL(\cG_{i-1})+\ELL(\cG_{i}) + \ELL(\cG_{i+1}) = \frac{3\ELL(\cG)}{2N+1} \sim \frac{\ELL(Q)}{N},
		\end{equation*}
		and also $\ell(\cB)=|\pi_0(z_i)-\pi_0(x)|$.
		
		Since $\cG_i$ is a leftist rectangle, it follows immediately from the definitions of leftist rectangles and leftmost points that
		\begin{equation}\label{eq:tildePinoE2}
		\text{int}(\cB)\cap E = \varnothing,
		\end{equation}
		so that $\text{int}(\cB)$ is a good (open) rectangle.
		\begin{figure}
			\def\svgwidth{3cm}
\begingroup%
  \makeatletter%
  \providecommand\color[2][]{%
    \errmessage{(Inkscape) Color is used for the text in Inkscape, but the package 'color.sty' is not loaded}%
    \renewcommand\color[2][]{}%
  }%
  \providecommand\transparent[1]{%
    \errmessage{(Inkscape) Transparency is used (non-zero) for the text in Inkscape, but the package 'transparent.sty' is not loaded}%
    \renewcommand\transparent[1]{}%
  }%
  \providecommand\rotatebox[2]{#2}%
  \newcommand*\fsize{\dimexpr\f@size pt\relax}%
  \newcommand*\lineheight[1]{\fontsize{\fsize}{#1\fsize}\selectfont}%
  \ifx\svgwidth\undefined%
    \setlength{\unitlength}{176.03997371bp}%
    \ifx\svgscale\undefined%
      \relax%
    \else%
      \setlength{\unitlength}{\unitlength * \real{\svgscale}}%
    \fi%
  \else%
    \setlength{\unitlength}{\svgwidth}%
  \fi%
  \global\let\svgwidth\undefined%
  \global\let\svgscale\undefined%
  \makeatother%
  \begin{picture}(1,2.085858)%
    \lineheight{1}%
    \setlength\tabcolsep{0pt}%
    \put(0,0){\includegraphics[width=\unitlength,page=1]{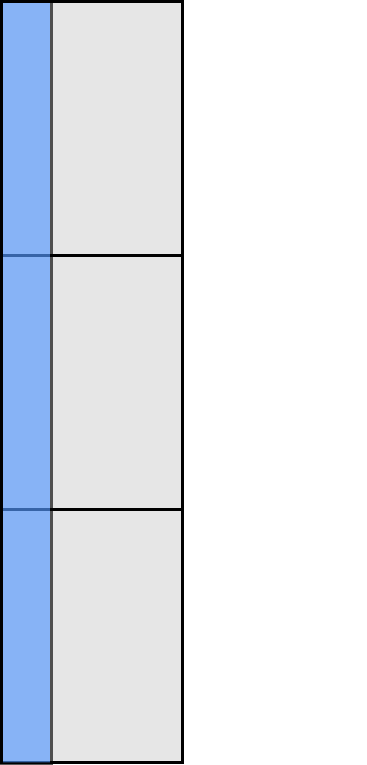}}%
    \put(0.56540633,0.32516552){\color[rgb]{0,0,0}\makebox(0,0)[lt]{\lineheight{1.25}\smash{\begin{tabular}[t]{l}$\cG_{i-1}$\end{tabular}}}}%
    \put(0.56540633,1.03396356){\color[rgb]{0,0,0}\makebox(0,0)[lt]{\lineheight{1.25}\smash{\begin{tabular}[t]{l}$\cG_i$\end{tabular}}}}%
    \put(0.56540633,1.73971835){\color[rgb]{0,0,0}\makebox(0,0)[lt]{\lineheight{1.25}\smash{\begin{tabular}[t]{l}$\cG_{i+1}$\end{tabular}}}}%
    \put(0,0){\includegraphics[width=\unitlength,page=2]{rectangletildePi.pdf}}%
    \put(0.18249146,1.25785929){\color[rgb]{0,0,0}\makebox(0,0)[lt]{\lineheight{1.25}\smash{\begin{tabular}[t]{l}$z_i$\end{tabular}}}}%
  \end{picture}%
\endgroup%

			\caption{The blue rectangle is $\cB$. In \lemref{lem:tildePiwide} we show that $\ell(\cB)\sim\ell(\cG)\sim\ell(Q)$.}\label{fig:rectangletildePi}
		\end{figure}
		\begin{lemma}\label{lem:tildePiwide}
			We have $|\pi_0(z_i)-\pi_0(x)|=\ell(\cB)\sim\ell(Q)$.
		\end{lemma}
		\begin{proof}
			Since $\cB\subset\cG$, it is clear that
			\begin{equation*}
			\ell(\cB)\le\ell(\cG)\sim\ell(Q),
			\end{equation*}
			so we only need to prove $\ell(\cB)\gtrsim\ell(\cG)\sim\ell(Q)$. See \figref{fig:rectangleP'} to get some intuition on why this is true. We give a formal argument below.
			
			Assume the contrary, so that $\ell(\cB)\le c\,\ell(\cG)$ for some small absolute constant $0<c<1$. We claim that if $0<c<1$ is chosen small enough, then
			\begin{equation}\label{eq:rectangle in cone}
				\cB\subset X(x,0.5J,A^{-1}\ELL(Q), A\ELL(Q)).
			\end{equation}
%
			To see that, observe that if $z\in\cB$, then 
			\begin{equation*}
			|\pi_0(z)-\pi_0(x)|\le \ell(\cB)\le c\ell(\cG)\sim c\ell(Q),
			\end{equation*}
			and also, since $\cB\subset\cG$,
			\begin{equation*}
			\frac{\ELL(\cG)}{2}\le |\pi_0^\perp(z)-\pi_0^\perp(x)| \le \frac{3\ELL(\cG)}{2}.
			\end{equation*}
			In particular, 
			$|\pi_0^\perp(z)-\pi_0^\perp(x)|\sim\ELL(\cG)\sim\ELL(Q)=\HH(J)^{-1}\ell(Q)$.
			It follows that 
			\begin{equation*}
			|\pi_0(z)-\pi_0(x)|\lesssim c\HH(J)|\pi_0^\perp(z)-\pi_0^\perp(x)|.
			\end{equation*}
			If $0<c<1$ is chosen small enough, we get that $z\in X(x,0.5J)$.
			
			Since 
			\begin{equation*}
			|x-z|\sim |\pi_0(z)-\pi_0(x)|+|\pi_0^\perp(z)-\pi_0^\perp(x)|\sim \ELL(Q),
			\end{equation*}
			we also have $z\in X(x,0.5J,A^{-1}\ELL(Q),A\ELL(Q))$ if $A$ is chosen large enough. This shows \eqref{eq:rectangle in cone}.
			\begin{figure}
				\def\svgwidth{8cm}
\begingroup%
  \makeatletter%
  \providecommand\color[2][]{%
    \errmessage{(Inkscape) Color is used for the text in Inkscape, but the package 'color.sty' is not loaded}%
    \renewcommand\color[2][]{}%
  }%
  \providecommand\transparent[1]{%
    \errmessage{(Inkscape) Transparency is used (non-zero) for the text in Inkscape, but the package 'transparent.sty' is not loaded}%
    \renewcommand\transparent[1]{}%
  }%
  \providecommand\rotatebox[2]{#2}%
  \newcommand*\fsize{\dimexpr\f@size pt\relax}%
  \newcommand*\lineheight[1]{\fontsize{\fsize}{#1\fsize}\selectfont}%
  \ifx\svgwidth\undefined%
    \setlength{\unitlength}{480.36810867bp}%
    \ifx\svgscale\undefined%
      \relax%
    \else%
      \setlength{\unitlength}{\unitlength * \real{\svgscale}}%
    \fi%
  \else%
    \setlength{\unitlength}{\svgwidth}%
  \fi%
  \global\let\svgwidth\undefined%
  \global\let\svgscale\undefined%
  \makeatother%
  \begin{picture}(1,1.2011657)%
    \lineheight{1}%
    \setlength\tabcolsep{0pt}%
    \put(0,0){\includegraphics[width=\unitlength,page=1]{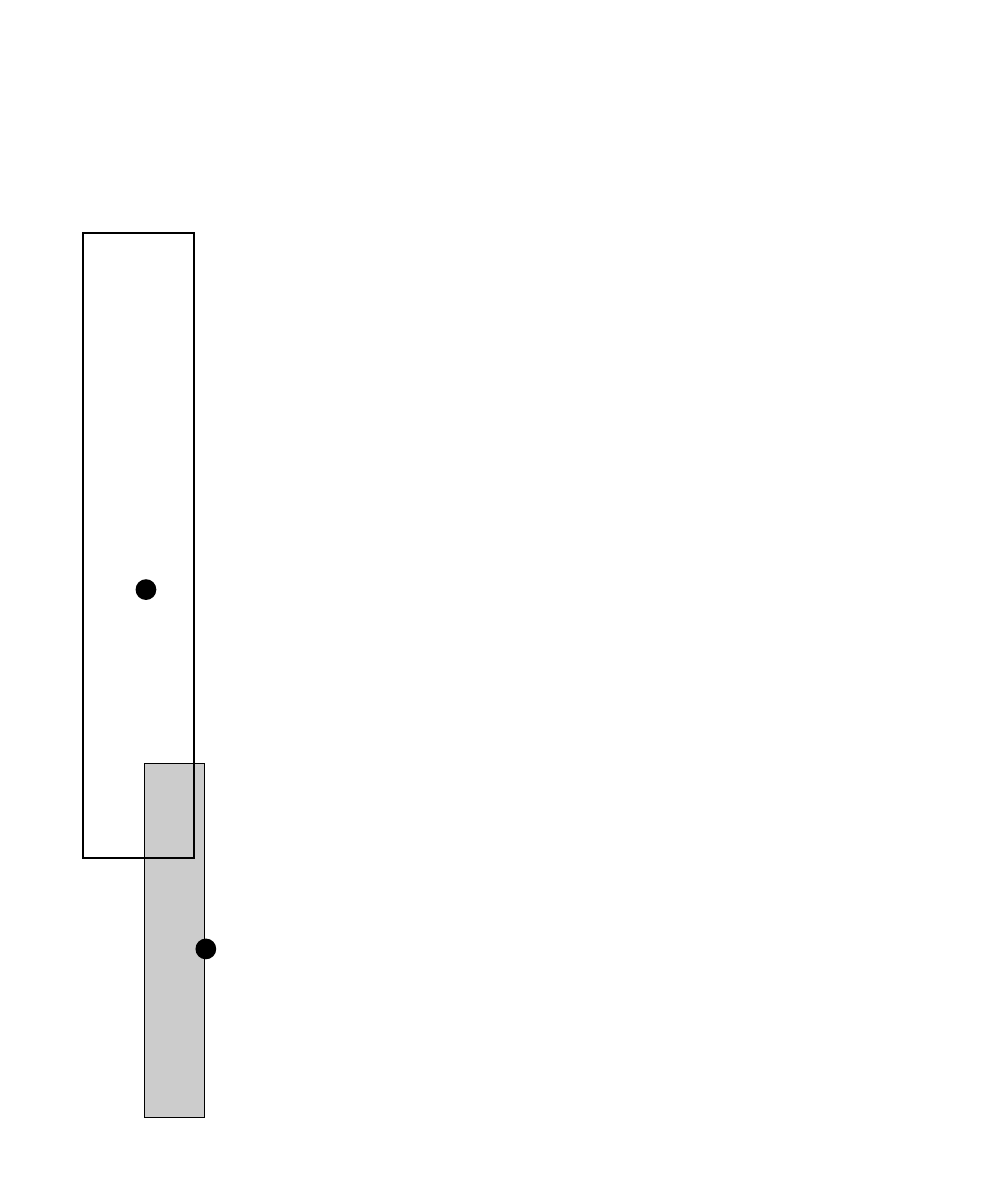}}%
    \put(0.09778237,0.59782026){\color[rgb]{0,0,0}\makebox(0,0)[lt]{\lineheight{1.25}\smash{\begin{tabular}[t]{l}$x$\end{tabular}}}}%
    \put(0.22855417,0.24405038){\color[rgb]{0,0,0}\makebox(0,0)[lt]{\lineheight{1.25}\smash{\begin{tabular}[t]{l}$y$\end{tabular}}}}%
    \put(0,0){\includegraphics[width=\unitlength,page=2]{rectangleP_.pdf}}%
    \put(0.69766336,1.11805354){\color[rgb]{0,0,0}\makebox(0,0)[lt]{\lineheight{1.25}\smash{\begin{tabular}[t]{l}$x$\end{tabular}}}}%
    \put(0.89469072,0.41204688){\color[rgb]{0,0,0}\makebox(0,0)[lt]{\lineheight{1.25}\smash{\begin{tabular}[t]{l}$y$\end{tabular}}}}%
    \put(0,0){\includegraphics[width=\unitlength,page=3]{rectangleP_.pdf}}%
    \put(0.83594918,0.51662254){\color[rgb]{0,0,0}\makebox(0,0)[lt]{\lineheight{1.25}\smash{\begin{tabular}[t]{l}$z_i$\end{tabular}}}}%
  \end{picture}%
\endgroup%

				\caption{On the left we see the full picture, on the right we zoom in on the dashed-border rectangle. The white rectangle is $\cR_Q$, the gray rectangle is $\cG$, the blue rectangle is $\cB$, the red double-truncated cone is $X(x,0.5J,A^{-1}\ELL(Q), A\ELL(Q))$. The red cone has an empty intersection with $E$ by \eqref{eq:empty cones}, whereas $\cB$ contains the point $z_i\in E$. Thus, $\cB$ cannot be fully contained in the red cone, which gives $\ell(\cB)\gtrsim\ell(Q)$.}\label{fig:rectangleP'}
			\end{figure}
		
			Recall that $X(x,0.5J,A^{-1}\ELL(Q), A\ELL(Q))\cap E=\varnothing$ by the assumption \eqref{eq:empty cones}. At the same time,
			$\cB$ contains $z_i\in E$. This contradicts \eqref{eq:rectangle in cone}. Hence,
			\begin{equation*}
			\ell(\cB)\ge c\ell(\cG)\sim \ell(Q).
			\end{equation*}
		\end{proof}
		
		\subsection{Big good rectangle $\mathcal{Y}$}
		\begin{figure}
			\def\svgwidth{13cm}
\begingroup%
  \makeatletter%
  \providecommand\color[2][]{%
    \errmessage{(Inkscape) Color is used for the text in Inkscape, but the package 'color.sty' is not loaded}%
    \renewcommand\color[2][]{}%
  }%
  \providecommand\transparent[1]{%
    \errmessage{(Inkscape) Transparency is used (non-zero) for the text in Inkscape, but the package 'transparent.sty' is not loaded}%
    \renewcommand\transparent[1]{}%
  }%
  \providecommand\rotatebox[2]{#2}%
  \newcommand*\fsize{\dimexpr\f@size pt\relax}%
  \newcommand*\lineheight[1]{\fontsize{\fsize}{#1\fsize}\selectfont}%
  \ifx\svgwidth\undefined%
    \setlength{\unitlength}{855.25854011bp}%
    \ifx\svgscale\undefined%
      \relax%
    \else%
      \setlength{\unitlength}{\unitlength * \real{\svgscale}}%
    \fi%
  \else%
    \setlength{\unitlength}{\svgwidth}%
  \fi%
  \global\let\svgwidth\undefined%
  \global\let\svgscale\undefined%
  \makeatother%
  \begin{picture}(1,1.09168863)%
    \lineheight{1}%
    \setlength\tabcolsep{0pt}%
    \put(0,0){\includegraphics[width=\unitlength,page=1]{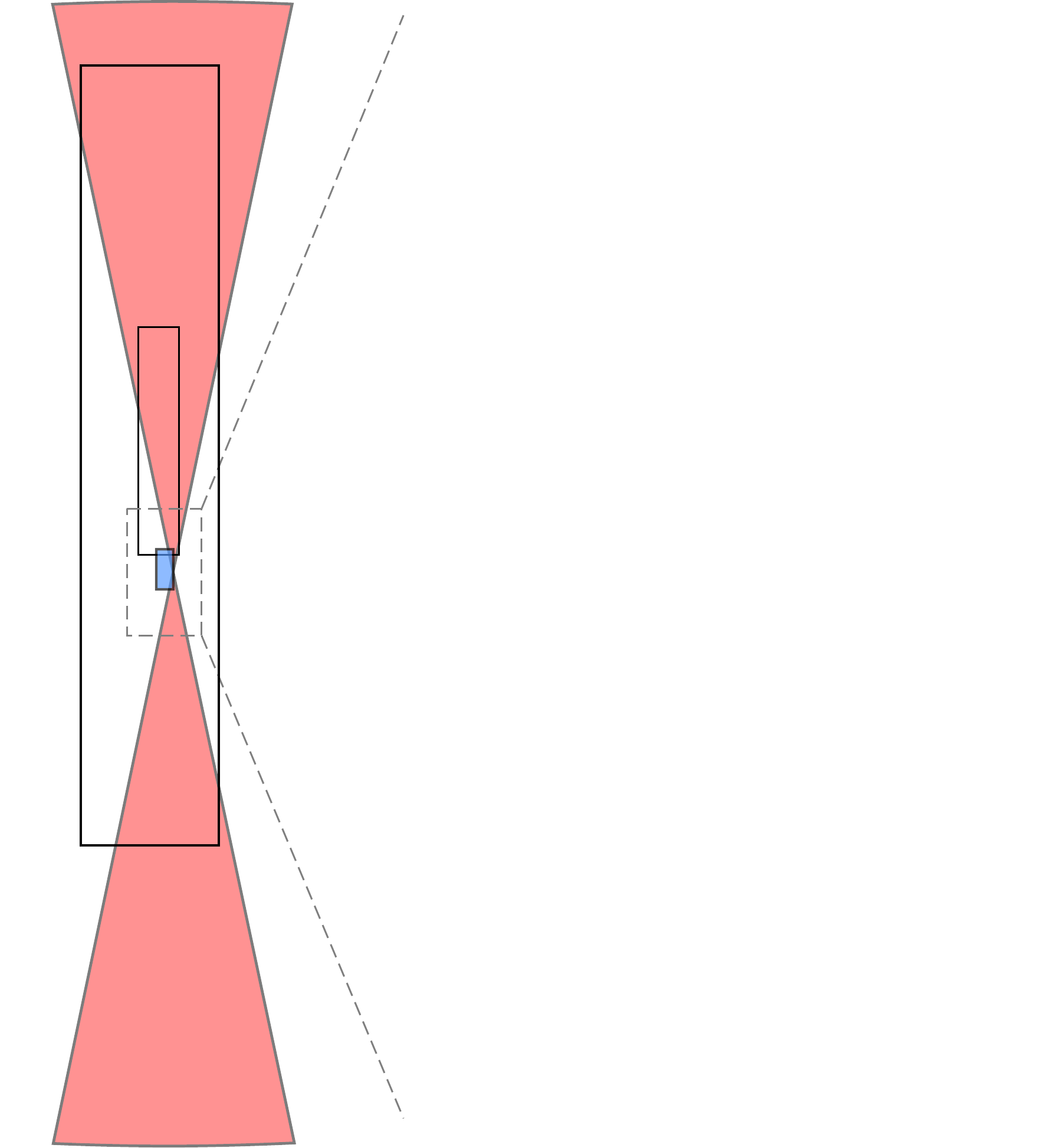}}%
    \put(0.17168873,0.54192865){\color[rgb]{0,0,0}\makebox(0,0)[lt]{\lineheight{1.25}\smash{\begin{tabular}[t]{l}$z_i$\end{tabular}}}}%
    \put(0,0){\includegraphics[width=\unitlength,page=2]{rectangleA.pdf}}%
    \put(0.08458017,0.63125741){\color[rgb]{0,0,0}\makebox(0,0)[lt]{\lineheight{1.25}\smash{\begin{tabular}[t]{l}$\cR_Q$\end{tabular}}}}%
    \put(-0.00116467,0.35474589){\color[rgb]{0,0,0}\makebox(0,0)[lt]{\lineheight{1.25}\smash{\begin{tabular}[t]{l}$A\cR_R$\end{tabular}}}}%
    \put(0,0){\includegraphics[width=\unitlength,page=3]{rectangleA.pdf}}%
    \put(0.78894802,0.55202081){\color[rgb]{0,0,0}\makebox(0,0)[lt]{\lineheight{1.25}\smash{\begin{tabular}[t]{l}$z_i$\end{tabular}}}}%
    \put(0,0){\includegraphics[width=\unitlength,page=4]{rectangleA.pdf}}%
  \end{picture}%
\endgroup%

			\caption{On the left we see the full picture, on the right we zoom in on the dashed-border rectangle. The small white rectangle is $\cR_Q$, the large white rectangle is $A\cR_R$, the blue rectangle is $\cB$, the narrow yellow rectangle is $\cA$, the red double-truncated cone is $X(z_i,0.5J,A^{-1}\ELL(Q), A^2\ELL(R))$.}\label{fig:rectangleA}
		\end{figure}
		Consider the rectangle $\mathcal{Y}$ defined as
		\begin{equation*}
		\cA\coloneqq \{z\in\R^2\,:\, \pi_0(z_i) - A^{-1}\ell(Q)\le \pi_0(z)\le\pi_0(z_i),\ |\pi_0^\perp(z)-\pi_0^\perp(z_i)|\le 2A\ELL(R) \},
		\end{equation*}
		see the yellow rectangle in \figref{fig:rectangleA}. Note that $\ell(\cA)=A^{-1}\,\ell(Q)$, $\ELL(\cA)=4A\ELL(R)$, and the mid-point of its right edge is $z_i$.
		
		Our plan is the following. First, we will show that $\cA$ is contained in the union of the good cone $X(z_i,0.5J,A^{-1}\ELL(Q), A^2\ELL(R))$ (the red cone in the figure) and the good rectangle $\cB$ (the blue rectangle in the figure). Since the interiors of these two have empty intersections with $E$, we will conclude that $\text{int}(\cA)\cap E=\varnothing$. This will give us $K\in \cK(R,\ell(Q))$ with $\pi_0(\cR_Q)\subset A^3K$, the desired gap in $\pi_{0}(A\cR_R\cap E)$.
		\begin{lemma}\label{lem:AinGX}
			If $A=A(C_0,M)$ is chosen big enough, then
			\begin{equation}\label{eq:AinGX}
			\mathrm{int}{(\cA)}\subset 	\mathrm{int}(\cB)\cup X(z_i,0.5J,A^{-1}\ELL(Q), A^2\ELL(R)).
			\end{equation}
		\end{lemma}
		\begin{proof}
			This is easy to believe in after looking at \figref{fig:rectangleA} for a minute or two, but for the sake of completeness, we provide the computations below. They are easier to follow keeping \figref{fig:rectangleA} in mind.
			
			Let
			\begin{align*}
			\cA_1 &\coloneqq \{z\in\R^2\,:\, \pi_0(z_i) - A^{-1}\ell(Q)< \pi_0(z)<\pi_0(z_i),\ |\pi_0^\perp(z)-\pi_0^\perp(z_i)|< \ELL(\cG_i) \},\\
			\cA_2 &\coloneqq \mathrm{int}(\cA)\setminus \cA_1,
			\end{align*}
			so that $\mathrm{int}(\cA) = \cA_1\cup\cA_2$.
			We claim that
			\begin{equation}\label{eq:AinG}
			\cA_1\subset\mathrm{int}(\cB),
			\end{equation}
			and 
			\begin{equation}\label{eq:AinX}
			\cA_2\subset X(z_i,0.5J,A^{-1}\ELL(Q), A^2\ELL(R)).
			\end{equation}
			
			First we prove \eqref{eq:AinG}. By \lemref{lem:tildePiwide}, we have $\ell(\cA_1)=A^{-1}\ell(Q)\le \ell(\cB)$, assuming $A$ big enough. Since $z_1$ lies on the right edges of both $\cA_1$ and $\cB$, this immediately gives $\pi_0(\cA_1)\subset \pi_0(\mathrm{int}(\cB))$. On the other hand, recall that $z_i\in\cG_i$ and
			\begin{equation*}
			\pi_{0}^\perp(\cB) = \pi_{0}^\perp(\cG_{i-1})\cup \pi_{0}^\perp(\cG_i)\cup\pi_{0}^\perp(\cG_{i+1}),
			\end{equation*}
			see \figref{fig:rectangletildePi}. It follows that
			\begin{equation*}
			\pi_{0}^\perp(\cA_1) = (\pi_0^\perp(z_i) - \ELL(\cG_i),\ \pi_0^\perp(z_i) + \ELL(\cG_i) )\subset \pi_{0}^\perp(\mathrm{int}(\cB)).
			\end{equation*}
			Since both $\cA_1$ and $\mathrm{int}(\cB)$ are open rectangles with sides parallel to the axes, we conclude that $\cA_1\subset\mathrm{int}(\cB)$.
			
			We move on to \eqref{eq:AinX}. First, observe that for $z\in\cA_2$ we have, by the definition of $\cA$,
			\begin{equation*}
			|z-z_i| \le (A^{-2}\ell(Q)^2 + 4A^2\ELL(R)^2)^{1/2} \le 3A\ELL(R),
			\end{equation*}
			and also, since $z\notin\cA_1$,
			\begin{equation*}
			|z-z_i| \ge  |\pi_0^\perp(z)-\pi_0^\perp(z_i)| \ge \ELL(\cG_i) = \frac{\ELL(\cG)}{2N+1}\overset{\eqref{eq:ELLPi}}{\sim} \frac{\ELL(Q)}{N}\overset{\eqref{eq:Ndef}}{\sim}\frac{\ELL(Q)}{MC_0}
			\end{equation*}
			Thus, assuming $A=A(M,C_0)$ large enough, we have
			\begin{equation*}
			z\in B(z_i, A^2\ELL(R))\setminus B(z_i, A^{-1}\ELL(Q)).
			\end{equation*}
			It remains to show $z\in X(z_i, 0.5J)$. Note that
			\begin{multline*}
			|\pi_{0}(z)-\pi_{0}(z_i)|\le A^{-1}\ell(Q) = A^{-1}\HH(J)\ELL(Q) \\
			= MC_0A^{-1}\HH(J)\frac{\ELL(Q)}{MC_0}\lesssim  MC_0A^{-1}\HH(J)\,|\pi_0^\perp(z)-\pi_0^\perp(z_i)|.
			\end{multline*}
			Assuming $A=A(M,C_0)$ large enough, this gives $z\in X(z_i, 0.5J)$.
		\end{proof}
		 \begin{lemma}\label{lem:AE2empty}
		 	We have $\mathrm{int}(\cA)\cap E=\varnothing$.
		 \end{lemma}
		 \begin{proof}
		 	Recall that $z_i\in\cG\cap E$, and $\cG \subset A\cR_Q$. Thus, $z_i\in A\cR_Q\cap E$, and so we get from \eqref{eq:empty cones} that 
		 	\begin{equation*}
		 	X(z_i,0.5J,A^{-1}\ELL(Q), A^2\ELL(R))\cap E=\varnothing.
		 	\end{equation*}
		 	We also have $\text{int}(\cB)\cap E=\varnothing$ by \eqref{eq:tildePinoE2}. Hence, it follows from \eqref{eq:AinGX} that
		 	\begin{equation*}
		 	\text{int}(\cA)\cap E = \varnothing.
		 	\end{equation*}
		 \end{proof}
	
		\subsection{Mind the gap} We are finally ready to find the gap $K\in\cK(R,\ell(Q))$ with $\pi_0(\cR_Q)\subset A^3K$.
		
		First, note that $z_i\in A\cR_Q\subset A\cR_R$. Since $\ELL(\cA) = 4A\ELL(R)$ and $z_i$ is the mid-point of the right edge of $\cA$, it follows that
		\begin{equation*}
		\{z\in A\cR_R\, :\,  \pi_0(z)\in\pi_{0}(\text{int}(\cA))\}\subset \text{int}(\cA)
		\end{equation*}
		Together with \lemref{lem:AE2empty}, this gives
		\begin{equation*}
		\{z\in A\cR_R\cap E\, :\,  \pi_0(z)\in\pi_{0}(\text{int}(\cA))\}\subset \text{int}(\cA)\cap E = \varnothing.
		\end{equation*}
		Hence,
		\begin{equation*}
		\pi_0(A\cR_R\cap E)\cap \pi_0(\text{int}(\cA))=\varnothing.
		\end{equation*}
		This means that the open interval $\pi_0(\text{int}(\cA)) = (\pi_{0}(z_i) - A^{-1}\ell(Q),\, \pi_{0}(z_i))$ is contained in some gap $K\in\cK(R)$. We have
		\begin{equation*}
		\HH(K)\ge \HH(\pi_0(\text{int}(\cA))) = A^{-1}\ell(Q).
		\end{equation*} 
		
		Note that $x,z_i\in A\cR_R\cap E$. Thus, $\pi_0(x),\pi_{0}(z_i)\notin K$, and also $\pi_{0}(z_i)$ lies on the right end-point of $K$. By \lemref{lem:tildePiwide}
		\begin{equation*}
		\pi_0(z_i)-\pi_0(x) = \ell(\cB) > A^{-1}\ell(Q) = \HH(\pi_0(\text{int}(\cA))),
		\end{equation*}
		so that
		\begin{equation*}
		\pi_{0}(x)\le 	\pi_0(z_i)-\HH(\pi_0(\text{int}(\cA))).
		\end{equation*}
		This means that $\pi_{0}(x)$ lies ``to the left'' of the interval $\pi_0(\text{int}(\cA))$, and in consequence, ``to the left'' of the gap $K$.	Since $\pi_0(z_i)$ is the right end-point of $K$, it follows from \lemref{lem:tildePiwide} that
		\begin{equation*}
		\HH(K)\le|\pi_0(x)-\pi_0(z_i)|=\ell(\cB)\sim\ell(Q).
		\end{equation*}
		So we have $A^{-1}\ell(Q)\le \HH(K)\lesssim\ell(Q)$. In particular, $K\in \cK(R,\ell(Q))$. 
		
		Finally, we have
		\begin{equation*}
		\dist(\pi_0(\cR_Q),K)\le\dist(\pi_0(x),K)\le |\pi_0(x)-\pi_0(z_i)|\lesssim \ell(Q)\le A\HH(K),
		\end{equation*}
		and so $\pi_0(\cR_Q)\subset A^3K$. This finishes the proof of \lemref{lem:why-good-rectangle}, and of the key geometric lemma.
		\appendix
		\section{Proof of Corollary \ref{cor:changtolsa}}\label{sec:appendix}
		In this section we prove Corollary \ref{cor:changtolsa}, which we repeat below for reader's convenience.
			\begin{cor}
			Let $E\subset\R^2$ and $G\subset\TT$ be as in \thmref{thm:main thm}, and let $\mu = \HH|_E$. Then,
			\begin{equation*}
				\int_{\R^2} \int_0^\infty \frac{\mu(X(x,G^\perp,r))}{r}\, \frac{dr}{r}d\mu(x) \lesssim M\HH(G)\mu(E),
			\end{equation*}
			where $G^\perp=G+1/4$.
		\end{cor}
			\begin{proof}
			If the set $G$ is open, then we can immediately apply \propref{prop:changtolsa} to estimate
			\begin{multline}\label{eq:10}
				\int_{\R^2} \int_0^\infty \frac{\mu(X(x,G^\perp,r))}{r}\, \frac{dr}{r}d\mu(x) \lesssim \int_{G}\|\pi_\theta\mu\|_2^2\, d\theta = \int_G \int_\R |\pi_\theta \mu(x)|^2\, dx\, d\theta\\
				\overset{\eqref{eq:projbdd}}{\le}  M\int_G\int_{\R} \pi_\theta \mu(x)\, dx\, d\theta = M\HH(G) \mu(E),
			\end{multline}
			which is the desired inequality.
			
			The general case will follow from the classical Besicovitch projection theorem and approximation. Suppose that $G$ is not open. Note that the assumption \eqref{eq:projbdd} implies that $\HH(\pi_\theta(E))>0$ for all $\theta\in G$, and even $\HH(\pi_\theta(F))>0$ for all $F\subset E$ with $\HH(F)>0$. Since $\HH(G)>0$, we get from the classical Besicovitch projection theorem, \thmref{thm:BesFed}, that $E$ is rectifiable, so that
			\begin{equation*}
				E = \bigcup_{i=1}^\infty \Gamma_i\cup Z,
			\end{equation*}
			where $\Gamma_i$ is a measurable subset of a graph of a $C^1$-function, and $\HH(Z)=0$. For $N\ge 1$ set
			\begin{equation*}
				E_N \coloneqq \bigcup_{i=1}^N \Gamma_i,
			\end{equation*}
			and $\mu_N = \HH|_{E_N}$.
			
			Fix $\theta\in G$. Since $\|\pi_\theta\mu\|_\infty\le M$, we have that for each $i\in\mathbb{N}$ and $\HH$-a.e. point $x\in \Gamma_i$ the line tangent to $\Gamma_i$ at $x$ cannot be perpendicular to $\ell_\theta$, and even
			\begin{equation*}
				\measuredangle(T_{x}\Gamma_i, \ell_\theta)\le \frac{\pi}{2} - CM^{-1}
			\end{equation*}
			for some absolute constant $0<C<1$. Hence, if $|\theta'-\theta|\le c M^{-1}$ for some small absolute constant $0<c<1$, then we have
			\begin{equation*}
				\measuredangle(T_{x}\Gamma_i, \ell_{\theta'})\le \frac{\pi}{2} - C'M^{-1}.
			\end{equation*}
			It follows that if $|\theta'-\theta|\le c M^{-1}$, then for any $i\in \mathbb{N}$ we have $\|\pi_{\theta'}\HH|_{\Gamma_i}\|_\infty\lesssim M$. Thus,
			\begin{equation*}
				\|\pi_{\theta'}\mu_N\|_\infty \le \sum_{i=1}^N \|\pi_{\theta'}\HH|_{\Gamma_i}\|_\infty\lesssim NM.
			\end{equation*}
			
			By the outer regularity of Lebesgue measure, there exists a sequence of open sets $G_k\supset G$ such that
			\begin{equation*}
				\HH(G_k\setminus G)\le \frac{1}{k}.
			\end{equation*}
			Without loss of generality we may assume that each $G_k$ is contained in a $c M^{-1}$-neighbourhood of $G$, so that for all $\theta\in G$ we have $\|\pi_{\theta}\mu_N\|_\infty \le \|\pi_{\theta}\mu\|_\infty \le M$ and for all $\theta\in G_k\setminus G$ we have $\|\pi_{\theta}\mu_N\|_\infty \lesssim NM$. Then, repeating the computation from \eqref{eq:10} yields
			\begin{multline}\label{eq:11}
				\int_{\R^2} \int_0^\infty \frac{\mu_N(X(x,G_k,r))}{r}\, \frac{dr}{r}d\mu_N(x) \lesssim \int_{G_k}\|\pi_\theta\mu_N\|_2^2\, d\theta\\
				\le  M\HH(G) \mu_N(E) +  MN\HH(G_k\setminus G) \mu_N(E).
			\end{multline}
			Note that $\mu_N(X(x,G,r))\le \liminf_k \mu_N(X(x,G_k,r))$, and so by Fatou's lemma
			\begin{multline}\label{eq:12}
				\int_{\R^2} \int_0^\infty \frac{\mu_N(X(x,G,r))}{r}\, \frac{dr}{r}d\mu_N(x) \le \int_{\R^2} \int_0^\infty \liminf_{k\to\infty}\frac{\mu_N(X(x,G_k,r))}{r}\, \frac{dr}{r}d\mu_N(x)\\
				\le \liminf_{k\to\infty} \int_{\R^2}\int_0^\infty \frac{\mu_N(X(x,G_k,r))}{r}\, \frac{dr}{r}d\mu_N(x)\\
				\lesssim \liminf_{k\to\infty} \big( M\HH(G) \mu_N(E) +  MN\HH(G_k\setminus G) \mu(E)\big)\\
				=  M\HH(G) \mu_N(E)\le  M\HH(G) \mu(E).
			\end{multline}
			
			Now, fix $0<r<\infty$. We claim that
			\begin{equation*}
				f_N(r)\coloneqq\int_{\R^2}\mu_N(X(x,G,r))\, d\mu_N(x)\xrightarrow{N\to\infty}\int_{\R^2}\mu(X(x,G,r))\, d\mu(x)\eqqcolon f(r).
			\end{equation*}
			Indeed, we have
			\begin{multline*}
				|f(r)-f_N(r)|=\int_{\R^2}\mu(X(x,G,r))\, d\mu(x) - \int_{\R^2}\mu_N(X(x,G,r))\, d\mu_N(x)\\ 
				= \int_{E\setminus E_N}\mu(X(x,G,r))\, d\mu(x) - \int_{E_N}\mu_N(X(x,G,r))-\mu(X(x,G,r))\, d\mu_N(x)\\
				\le \mu(E)\cdot \mu(E\setminus E_N) + \mu(E_N)\cdot \mu(E\setminus E_N)\xrightarrow{N\to\infty} 0.
			\end{multline*}
			Hence, by Fatou's lemma and Fubini's theorem
			\begin{multline*}
				\int_{\R^2} \int_0^\infty \frac{\mu(X(x,G,r))}{r}\, \frac{dr}{r}d\mu(x) = \int_0^\infty f(r)\frac{dr}{r^2} = \int_0^\infty \liminf_{N\to\infty}f_N(r)\frac{dr}{r^2}\\
				\le \liminf_{N\to\infty} \int_0^\infty f_N(r)\frac{dr}{r^2} = \liminf_{N\to\infty}\int_{\R^2} \int_0^\infty \frac{\mu_N(X(x,G,r))}{r}\, \frac{dr}{r}d\mu_N(x)\\
				\overset{\eqref{eq:12}}{\lesssim}\liminf_{N\to\infty} M\HH(G) \mu(E) =  M\HH(G) \mu(E).
			\end{multline*}	
		\end{proof}
\newcommand{\etalchar}[1]{$^{#1}$}

\end{document}